\documentclass[a4paper,11pt]{amsart}

\usepackage[latin1]{inputenc}
\usepackage[cyr]{aeguill}
\usepackage[english]{babel}
\usepackage{tikz,hyperref,array}
\usepackage{graphicx}

\usepackage[T1]{fontenc}
\usepackage{amsmath,bbold}
\usepackage{amsfonts}
\usepackage{amssymb}
\usepackage{amsthm}
\usepackage{enumitem}
\usepackage{graphicx,color}
\usepackage{footnote}
\usepackage{newtxtext}
\usepackage{dsfont,subcaption}
\usepackage[dvipsnames]{xcolor}
\definecolor{darkgreen}{rgb}{0,0.4,0}
\definecolor{MyDarkBlue}{rgb}{0,0.08,0.85}
\definecolor{BrickRed}{rgb}{0.8,0.08,0}
\usetikzlibrary{arrows.meta}

\usepackage{multirow}

\hypersetup{colorlinks=true,    linkcolor=blue,
citecolor=red,       filecolor=BrickRed,       urlcolor=darkgreen
}

\newcommand{\B}{\mathcal{B}}

\newcommand{\G}{\mathcal{G}}

\newcommand{\R}{\mathcal{R}}
\newcommand{\EE}{\mathbb{E}}
\newcommand{\NN}{\mathbb{N}}
\newcommand{\PP}{\mathbb{P}}
\newcommand{\RR}{\mathbb{R}}
\newcommand{\ZZ}{\mathbb{Z}}
\renewcommand{\geq}{\geqslant}
\renewcommand{\leq}{\leqslant}
\renewcommand{\epsilon}{\varepsilon}

\def \card{\mbox{Card}}
\def \tS{\widetilde{S}}
\def \txi{\widetilde{\xi}}

\def \BN{\mathcal{BN}}

\newtheorem{thm}{Theorem}

\newtheorem{lem}[thm]{Lemma}
\newtheorem{rem}[thm]{Remark}

\title[Returns to the coordinate axes for conditioned random walks]{Asymptotics of returns to the coordinate axes for conditioned planar simple random walks}

\author[R.~Garbit]{Rodolphe Garbit}
\address{Universit\'e d'Angers, CNRS, Laboratoire Angevin de Recherche en Math\'ematiques, SFR MATHSTIC, 49000 Angers, France}
\email{rodolphe.garbit@univ-angers.fr}
\email{raschel@math.cnrs.fr}
\author[K.~Raschel]{Kilian Raschel}
\thanks{RG and KR are supported in part by the project RAWABRANCH (\href{https://anr.fr/Project-ANR-23-CE40-0008}{ANR-23-CE40-0008}), funded by the
French National Research Agency, and from Centre Henri Lebesgue (\href{https://anr.fr/ProjetIA-11-LABX-0020}{ANR-11-LABX-0020-0}).}


\date{\today}
\begin{document}

\begin{abstract}
In this paper we study the number of returns to the coordinate axes for two-dimensional nearest-neighbour random walks. While one-dimensional results on returns are classical, much less is known in higher dimensions. We analyse the asymptotic behaviour of returns under several natural conditionings: the unconditioned walk, bridges, meanders, and non-negative bridges (or excursions). Our main results characterize the limiting distributions under appropriate rescaling. The resulting one-dimensional marginals may be half-normal, Rayleigh, geometric, negative binomial, or certain mixtures thereof. In most situations, the coordinates are asymptotically independent; however, there are notable exceptions for the meander case, depending on the drift. The proofs rely on conditioning on the numbers of horizontal and vertical steps, which restores a form of independence and reduces the problem to one-dimensional estimates via binomial convolution and Bernstein-type approximations.
\end{abstract}

\maketitle 

%\tableofcontents

%%%%%%%%%%%%%%%%%%%%%%%%%%%%%%%%%%%%%%%%%%%%%%%%%%%%%%%%%%%%%%%%%%%%%%%%%
\section{Introduction and main results}

\subsection{Context} Given a one-dimensional random walk on the lattice $\mathbb Z$, its returns to the origin $0$ are among the most extensively studied statistics. They are closely connected to the fluctuations of the walk, to its recurrent or transient nature, and to the associated Green function (via the number of visits). Classical references such as Feller \cite{Fel68,Fel71} or Spitzer \cite{Spi-76} provide fundamental results, including exact identities in terms of related statistics (such as first-passage times) and limit theorems as the walk length tends to infinity. Broadly speaking, there are four main settings: the classical unconditioned case (where only the starting point is fixed), the bridge case (with both start and end points specified), the meander case (where the walk is conditioned to remain non-negative), and the non-negative bridge case (conditioned walk with fixed endpoint). See Figure~\ref{fig:RW_dim1}.

More recently, these questions have been revisited and refined through the lens of analytic combinatorics; see, for instance, \cite{BaFl-02,FlaSed09,BaWa-14,Wa-16,Wa-20,BaKuWaWa-24,BaKuWa-24}. The main contributions include explicit expressions for the generating functions of returns to $0$, the study of variants such as reflection or absorption models, connections with enumerative combinatorics, Gibbs measures, and more. Altogether, the one-dimensional case is now well understood, although, as we will see, the corresponding analysis for meanders was not yet fully developed.
%It is intuitively clear that the drift of the walk should influence the limit theorems.

In higher dimensions $d$, the behaviour of random walks becomes significantly richer, and the above question naturally extends in several directions. One may study the returns to the origin in $\mathbb Z^d$, or the returns to the hyperplanes defined by a single coordinate being zero, corresponding to the coordinate axes in two dimensions. It is precisely this variant that we will focus on in the present work.

Further motivation for the multidimensional setting comes from statistical physics, notably from models of interacting directed polymers. For example, the model of three non-crossing paths studied in \cite{TaOwRe-16} can be encoded as a two-dimensional walk confined to the quarter plane by tracking the distances between the paths. In this representation, boundary contacts correspond to returns to the axes. The associated partition function describes the statistics of walks of length $n$, and its exponential growth rate defines the free energy, which captures the large-scale behaviour and phase structure of the model.
See Figure~\ref{fig:RW_dim2}, left display, for a related model.

\begin{figure}
\begin{center}
\begin{subfigure}{0.45\textwidth}
  \begin{tikzpicture}[
    x=5.5mm,
    y=5.5mm,
  ]

    % Grid
    \draw[thin]
      \foreach \x in {0, ..., 10} {
        (\x, -3) -- (\x, 6)
      }
      \foreach \y in {-2, ..., 5} {
        (0, \y) -- (11, \y)
      }
    ;  

    % Coordinate axes
    \begin{scope}[
      thick,
      ->,
      >={Stealth[]},
    ]
      \draw (0, -3.5) -- (0, 6.5);
      \draw (0, 0) -- (11.5, 0);
    \end{scope}
  \draw[very thick] (0,0) -- (1,1) -- (2,2) -- (3,1) -- (4,0) -- (5,0) -- (6,-1) -- (7,-2) -- (8,-1) -- (9,-2) -- (10,-2) ;
\filldraw[blue] (0,0) circle (2.5pt);
\filldraw[blue] (1,1) circle (2.5pt);
\filldraw[blue] (2,2) circle (2.5pt);
\filldraw[blue] (3,1) circle (2.5pt);
\filldraw[red] (4,0) circle (3.5pt);
\filldraw[blue] (5,0) circle (2.5pt);
\filldraw[blue] (6,-1) circle (2.5pt);
\filldraw[blue] (7,-2) circle (2.5pt);
\filldraw[blue] (8,-1) circle (2.5pt);
\filldraw[blue] (9,-2) circle (2.5pt);
\filldraw[blue] (10,-2) circle (2.5pt);
    \end{tikzpicture}
    \caption{Unconditioned walk/path}
    \label{subfig:A}
    \end{subfigure}
    \begin{subfigure}{0.45\textwidth}
   \begin{tikzpicture}[
    x=5.5mm,
    y=5.5mm,
  ]

    % Grid
    \draw[thin]
      \foreach \x in {0, ..., 10} {
        (\x, -3) -- (\x, 6)
      }
      \foreach \y in {-2, ..., 5} {
        (0, \y) -- (11, \y)
      }
    ;  

    % Coordinate axes
    \begin{scope}[
      thick,
      ->,
      >={Stealth[]},
    ]
      \draw (0, -3.5) -- (0, 6.5);
      \draw (0, 0) -- (11.5, 0);
    \end{scope}
  \draw[very thick] (0,0) -- (1,-1) -- (2,-2) -- (3,-1) -- (4,0) -- (5,1) -- (6,2) -- (7,1) -- (8,2) -- (9,1) -- (10,0) ;
\filldraw[blue] (0,0) circle (2.5pt);
\filldraw[blue] (1,-1) circle (2.5pt);
\filldraw[blue] (2,-2) circle (2.5pt);
\filldraw[blue] (3,-1) circle (2.5pt);
\filldraw[red] (4,0) circle (3.5pt);
\filldraw[blue] (5,1) circle (2.5pt);
\filldraw[blue] (6,2) circle (2.5pt);
\filldraw[blue] (7,1) circle (2.5pt);
\filldraw[blue] (8,2) circle (2.5pt);
\filldraw[blue] (9,1) circle (2.5pt);
\filldraw[red] (10,0) circle (3.5pt);
    \end{tikzpicture}
    \caption{Bridge}
    \label{subfig:B}
    \end{subfigure}
    \begin{subfigure}{0.45\textwidth}
  \begin{tikzpicture}[
    x=5.5mm,
    y=5.5mm,
  ]

    % Grid
     \draw [fill=white,gray!20,draw opacity=1] (0,0) rectangle (11,-3);
    \draw[thin]
      \foreach \x in {0, ..., 10} {
        (\x, -3) -- (\x, 6)
      }
      \foreach \y in {-2, ..., 5} {
        (0, \y) -- (11, \y)
      }
    ;  

    % Coordinate axes
    \begin{scope}[
      thick,
      ->,
      >={Stealth[]},
    ]
   
      \draw (0, -3.5) -- (0, 6.5);
      \draw (0, 0) -- (11.5, 0);
    \end{scope}
  \draw[very thick] (0,0) -- (1,1) -- (2,0) -- (3,0) -- (4,0) -- (5,1) -- (6,0) -- (7,1) -- (8,2) -- (9,3) -- (10,4) ;
  
\filldraw[blue] (0,0) circle (2.5pt);
\filldraw[blue] (1,1) circle (2.5pt);
\filldraw[red] (2,0) circle (3.5pt);
\filldraw[blue] (3,0) circle (2.5pt);
\filldraw[blue] (4,0) circle (2.5pt);
\filldraw[blue] (5,1) circle (2.5pt);
\filldraw[red] (6,0) circle (3.5pt);
\filldraw[blue] (7,1) circle (2.5pt);
\filldraw[blue] (8,2) circle (2.5pt);
\filldraw[blue] (9,3) circle (2.5pt);
\filldraw[blue] (10,4) circle (2.5pt);
    \end{tikzpicture}
    \caption{Meander}
    \label{subfig:C}
    \end{subfigure}
    \begin{subfigure}{0.45\textwidth}
  \begin{tikzpicture}[
    x=5.5mm,
    y=5.5mm,
  ]

    % Grid
     \draw [fill=white,gray!20,draw opacity=1] (0,0) rectangle (11,-3);
    \draw[thin]
      \foreach \x in {0, ..., 10} {
        (\x, -3) -- (\x, 6)
      }
      \foreach \y in {-2, ..., 5} {
        (0, \y) -- (11, \y)
      }
    ;  

    % Coordinate axes
    \begin{scope}[
      thick,
      ->,
      >={Stealth[]},
    ]
   
      \draw (0, -3.5) -- (0, 6.5);
      \draw (0, 0) -- (11.5, 0);
    \end{scope}
  \draw[very thick] (0,0) -- (1,1) -- (2,2) -- (3,3) -- (4,2) -- (5,1) -- (6,0) -- (7,0) -- (8,1) -- (9,1) -- (10,0) ;
  
\filldraw[blue] (0,0) circle (2.5pt);
\filldraw[blue] (1,1) circle (2.5pt);
\filldraw[blue] (2,2) circle (2.5pt);
\filldraw[blue] (3,3) circle (2.5pt);
\filldraw[blue] (4,2) circle (2.5pt);
\filldraw[blue] (5,1) circle (2.5pt);
\filldraw[red] (6,0) circle (3.5pt);
\filldraw[blue] (7,0) circle (2.5pt);
\filldraw[blue] (8,1) circle (2.5pt);
\filldraw[blue] (9,1) circle (2.5pt);
\filldraw[red] (10,0) circle (3.5pt);
    \end{tikzpicture}
    \caption{Excursion}
    \label{subfig:D}
    \end{subfigure}
    \caption{Number of returns to $0$ (in red, with slightly larger filled circles)\ for the simple random walk with steps in $\{-1,0,1\}$ (often called a Motzkin walk in combinatorics)\ in the four main settings considered in this work: top left (Figure~\ref{subfig:A}), unconditioned walk on $\mathbb Z$; top right (Figure~\ref{subfig:B}), bridge conditioned to end at $0$; bottom left (Figure~\ref{subfig:C}), meander conditioned to stay non-negative; bottom right (Figure~\ref{subfig:D}), excursion conditioned to stay non-negative and to end at $0$. Here, the term ``return'' is used in the strict sense: the walk must leave the axis before returning to it. The number of returns is closely related to the total number of contacts with the axis, where every visit is counted.}
    \label{fig:RW_dim1}
\end{center}    
\end{figure}

The paper \cite{BeOwRe-19} initiated a systematic study of the number of contacts for small-step random walks in the quarter plane. Introducing a five-variable generating function that records the $x$- and $y$-coordinates of the endpoint, the length, and the number of contacts with the horizontal and vertical axes, the authors ask whether this series can be expressed explicitly and whether its nature (algebraic, differential, etc.) can be characterised. In the terminology used above, their results address the non-negative bridge case in several two-dimensional settings. Let us highlight two examples.

The first is the so-called diagonal model (jumps to the four directions $(\pm1,\pm1))$, where the independence of the coordinates allows one to write the number of returns to zero as the sum of two independent one-dimensional quantities. A second example is the simple random walk (jumps to the four directions $(\pm1,0),(0,\pm1)$). Here the coordinates are no longer independent, and the semi-explicit expression obtained in \cite{BeOwRe-19} already indicates that, even for this very classical model, the analysis of contacts (or returns to zero) is significantly more delicate than in the diagonal case. Moreover, the expression provided in \cite{BeOwRe-19} is difficult to analyse asymptotically as $n\to\infty$. Nevertheless, one would naturally expect some form of asymptotic independence to re-emerge in this regime, making it possible to derive limit theorems from the one-dimensional case.

\begin{figure}[ht!]
\begin{center}
  \begin{tikzpicture}[
    x=5.5mm,
    y=5.5mm,
  ]

    % Grid
    \draw [fill=white,gray!20,draw opacity=1] (0,0) rectangle (11,-3);
    \draw[thin]
      \foreach \x in {0, ..., 10} {
        (\x, -3) -- (\x, 6)
      }
      \foreach \y in {-2, ..., 5} {
        (0, \y) -- (11, \y)
      }
    ;  

    % Coordinate axes
    \begin{scope}[
      semithick,
      ->,
      >={Stealth[]},
    ]
      \draw (0, -3.5) -- (0, 6.5);
      \draw (0, 0) -- (11.5, 0);
    \end{scope}
  \draw[very thick] (0,3) -- (1,4) -- (2,5) -- (3,4) -- (4,3);
 \draw[very thick] -- (4.05,3) -- (5.05,2);
 \draw[very thick] (5,2) -- (6,3) -- (7,2);
 \draw[very thick] (7.05,2) -- (8.05,1) -- (9.05,0) -- (10.05,1)  ;
  \draw[very thick,brown] (0,1) -- (1,2) -- (2,3) -- (3,2) -- (4,3);
  \draw[very thick,brown] (3.95,3) -- (4.95,2);
  \draw[very thick,brown] (5,2) -- (6,1) -- (7,2);
  \draw[very thick,brown] (6.95,2) -- (7.95,1) -- (8.95,0) -- (9.95,1) ;
\filldraw[blue] (0,3) circle (2.5pt);
\filldraw[blue] (0,1) circle (2.5pt);
\filldraw[blue] (1,2) circle (2.5pt);
\filldraw[blue] (1,4) circle (2.5pt);
\filldraw[blue] (2,5) circle (2.5pt);
\filldraw[blue] (3,4) circle (2.5pt);
\filldraw[red] (4,3) circle (2.5pt);
\filldraw[red] (5,2) circle (2.5pt);
\filldraw[blue] (6,3) circle (2.5pt);
\filldraw[red] (7,2) circle (2.5pt);
\filldraw[red] (8,1) circle (2.5pt);
\filldraw[blue] (2,3) circle (2.5pt);
\filldraw[blue] (3,2) circle (2.5pt);
\filldraw[blue] (6,1) circle (2.5pt);
\filldraw[red] (9,0) circle (2.5pt);
\filldraw[red] (10,1) circle (2.5pt);
    \end{tikzpicture}\quad
  \begin{tikzpicture}[
    x=5.5mm,
    y=5.5mm,
  ]

    % Grid
     \draw [fill=white,gray!20,draw opacity=1] (-3,-3) rectangle (6,-1);
      \draw [fill=white,gray!20,draw opacity=1] (-3,-3) rectangle (-1,6);
    \draw[thin]
      \foreach \x in {-3, ..., 6} {
        (\x, -3) -- (\x, 6)
      }
      \foreach \y in {-3, ..., 6} {
        (-3, \y) -- (6, \y)
      }
    ;  

    % Coordinate axes
    \begin{scope}[
      semithick,
      ->,
      >={Stealth[]},
    ]
   
      \draw (-1, -3.5) -- (-1, 6.5);
      \draw (-3.5,-1) -- (6.5, -1);
    \end{scope}
  \draw[very thick] (-1,-1) -- (-1,0) -- (0,0) -- (0,-1) -- (1,-1)  -- (1,0) -- (1,1) -- (1,2) -- (1,3) -- (0,3) -- (-1,3) -- (-1,4) -- (0,4) -- (0,3) -- (-1,3) -- (0,3)-- (1,3) -- (2,3) -- (2,2) -- (3,2) -- (3,1) -- (4,1) -- (4,0) -- (4,-1) -- (3,-1);
  
\filldraw[blue] (-1,-1) circle (2.5pt);
\filldraw[blue] (-1,0) circle (2.5pt);
\filldraw[blue] (0,0) circle (2.5pt);
\filldraw[red] (0,-1) circle (2.5pt);
\filldraw[blue] (1,-1) circle (2.5pt);
\filldraw[blue] (1,0) circle (2.5pt);
\filldraw[blue] (1,1) circle (2.5pt);
\filldraw[blue] (1,2) circle (2.5pt);
\filldraw[blue] (1,3) circle (2.5pt);
\filldraw[blue] (0,3) circle (2.5pt);
\filldraw[red] (-1,3) circle (2.5pt);
\filldraw[blue] (-1,4) circle (2.5pt);
\filldraw[blue] (0,4) circle (2.5pt);
\filldraw[blue] (2,2) circle (2.5pt);
\filldraw[blue] (3,2) circle (2.5pt);
\filldraw[blue] (2,3) circle (2.5pt);
\filldraw[blue] (3,1) circle (2.5pt);
\filldraw[blue] (0,3) circle (2.5pt);
\filldraw[blue] (4,1) circle (2.5pt);
\filldraw[blue] (4,0) circle (2.5pt);
\filldraw[red] (4,-1) circle (2.5pt);
\filldraw[blue] (3,-1) circle (2.5pt);

    \end{tikzpicture}
    \caption{Left: two non-intersecting} simple random walks on $\NN_0$, exhibiting two types of boundary contacts: coincidences between the upper and lower paths, and contacts with the horizontal axis. Right: A simple random walk in the quarter plane; first returns to the coordinate axes are indicated in red.
    \label{fig:RW_dim2}
\end{center}    
\end{figure}

In this paper, our goal is to clarify this phenomenon and to provide an approach as elementary as possible to proving limit theorems for the number of returns to zero of the two-dimensional simple random walk.

\subsection{Main results}
Consider the nearest-neighbor random walk $\left(S_n\right)_{n\geq 0}$ on $\mathbb Z^2$, started at $S_0=(0,0)$ and with i.i.d.\ increments $\xi_n = S_n - S_{n-1}$ satisfying

\begin{minipage}{0.5\textwidth}
$$\begin{cases}
\PP(\xi_k=(1,0))=p_1,\\
\PP(\xi_k=(-1,0))=q_1,\\
\PP(\xi_k=(0,1))=p_2,\\
\PP(\xi_k=(0,-1))=q_2,
\end{cases}
$$\end{minipage}
\begin{minipage}{0.5\textwidth}
\begin{center}
 \begin{tikzpicture}[
    x=4.0mm,
    y=4.0mm,
  ]

    % Grid
     \draw (-3,-3) rectangle (3,-1);
      \draw  (-3,-3) rectangle (-1,3);
    \draw[thin]
      \foreach \x in {-1, ..., 2} {
        (2*\x-1, -3) -- (2*\x-1, 3)
      }
      \foreach \y in {-1, ..., 2} {
        (-3, 2*\y-1) -- (3, 2*\y-1)
      }
    ;  

    % Coordinate axes
    \begin{scope}[
      semithick,
      ->,
      >={Stealth[]},
    ]
   
      \draw (-1, -3.5) -- (-1, 3.5);
      \draw (-3.5,-1) -- (3.5, -1);
    \end{scope}
  \draw[ultra thick,red,->] (1,1) -- node[below right] {$p_1$} (3,1);
  \draw[ultra thick,red,->] (1,1) -- node[above left] {$q_1$} (-1,1);
 \draw[ultra thick,red,->] (1,1) -- node[above right] {$p_2$} (1,3);
 \draw[ultra thick,red,->] (1,1) -- node[below left] {$q_2$} (1,-1);
  
    \end{tikzpicture}
    \end{center}
    \end{minipage}
where $p_1,q_1, p_2, q_2$ are \emph{positive} real numbers that sum to $1$. We set
\begin{equation}
\label{defpeth}
h_i=p_i+q_i.
\end{equation}
The number $h_1$ is the probability of a horizontal step, and $h_2$ the probability of a vertical step. 
%Throughout this paper, we assume that the random walk is truly $2$-dimensional in the sense that $h_1>0$ and $h_2>0$.
We write $S^1_n$ and $S^2_n$ for the horizontal and vertical coordinates of $S_n$ and define the numbers of \emph{returns} to the two axes by
\begin{equation}
\label{eq:nindef}
N^i_n=\sharp\{ 1\leq k\leq n \mid S^i_k=0, S^i_{k-1}\not=0\}.
\end{equation}
Note that the word ``return'' is used here in a strict sense: the walk has to leave the axis before it returns.
In this work, we are interested in limit theorems concerning the pair
$(N^1_n,N^2_n)$ under various conditionings. Some of them involve the exit time 
$$\tau=\inf\{ n \geq 1 \mid  S^1_n<0 \mbox{ or } S^2_n<0\}$$
from the quadrant $\NN_0^2$. Here and throughout, we denote by $\NN_0:=\{0,1,2,\ldots\}$ the set of non-negative integers and by $\NN:=\{1,2,\ldots\}$ the set of positive integers.

The limit distributions featured in our results all stem from the following four fundamental distributions:
\begin{itemize}
\item \emph{Half-normal}: The  distribution on $[0,\infty)$ with cumulative distribution function
$$\phi^+(x)=\sqrt{\frac{2}{\pi}}\int_0^x e^{-\frac{t^2}{2}}\,dt.$$
\item \emph{Rayleigh}: The distribution on $[0,\infty)$ with cumulative distribution function
$$\R(x)=\int_0^x te^{-\frac{t^2}{2}}\,dt=1-e^{-\frac{x^2}{2}}.$$
\item \emph{Geometric} $\G(\alpha)$: The distribution on $\NN_0$ for which the probability of integer $k$ is $\alpha (1-\alpha)^k$.
\item \emph{Negative binomial} $\BN(2,1/2)$: The distribution on $\NN$ for which the probability of integer $k$ is $k/2^{k+1}$.
\end{itemize}

We are now ready to state our main results. As is standard, we use the symbol $\Rightarrow$ to denote convergence in distribution.

\begin{thm}[Unconditioned case]
\label{thm:2dunconditioned}
Set $a^i_n=\sqrt{h_i n}$ if $p_i=q_i$ and $a^i_n=1$ if $p_i\not= q_i$.
As $n\to\infty$, it holds that
$$\left(\frac{N^{1}_n}{a^1_n},\frac{N^{2}_n}{a^2_n} \right) \Rightarrow
(X_1,X_2),$$
where $X_1$ and $X_2$ are independent and the distribution of $X_i$ is 
\begin{itemize}
\item half-normal in case $p_i=q_i$,
\item geometric $\G\left(\vert p_i-q_i\vert/h_i\right)$ in case $p_i\not=q_i$.
\end{itemize}
\end{thm}

\begin{thm}[Bridge case] 
\label{thm:2dbridgecase}
Conditional on $S_n=(0,0)$, it holds that
$$\left(\frac{N^1_n}{\sqrt{h_1 n}},\frac{N^2_n}{\sqrt{h_2 n}} \right) \Rightarrow
(X_1,X_2),$$
where $X_1$ and $X_2$ are independent with standard Rayleigh distribution. (Here the limit is taken through $2\NN$.)
\end{thm}

\begin{thm}[Meander case]
\label{thm:2dmeandercase}
As $n\to\infty$ through $2\NN$, conditional on $\tau>n$, it holds that
$$(N^1_n, N^2_n) \Rightarrow (X_1,X_2),$$
where
\begin{itemize}
\item in case $p_i\geq q_i$,  $X_1$ and $X_2$ are independent and the distribution of $X_i$ is geometric $\G(p_i/h_i)$;
\item in case $p_1<q_1$ and $p_2\geq q_2$, $X_1$ and $X_2$ are independent; the distribution of $X_1$ is given by \eqref{eq:case<geq}, while $X_2$ follows a geometric distribution $\G(p_2/h_2)$; the case $p_2<q_2$ and $p_1\geq q_1$ is obtained by symmetry;
\item in case $p_i<q_i$ the random variables $X_1$ and $X_2$ are not independent, and their joint distribution is described in \eqref{eq:case<<}.
\end{itemize}
\end{thm}

Let us emphasize that Theorem~\ref{thm:2dmeandercase} is the most original result of this work; it provides a concrete example (in the case of meanders) featuring  limit distributions in which the coordinates may fail to be independent.
In the last case of Theorem~\ref{thm:2dmeandercase}, non-independence can be shown by the following elementary covariance computation:
\begin{equation*}
    \operatorname{Cov} (X_1,X_2) = \frac{\rho_1\rho_2}{(1+\rho_1\rho_2)^2} \frac{(\widetilde q_1-\widetilde p_1)^2(\widetilde q_2-\widetilde p_2)^2}{\widetilde q_1\widetilde q_2},
\end{equation*}
where we have noted $\widetilde p_i=1-\widetilde q_i=\frac{p_i}{p_i+q_i}$ and $\rho_i=2\sqrt{\widetilde p_i\widetilde q_i}$. Observe that this covariance is strictly positive.

\begin{thm}[Non-negative bridge]
\label{thm:2dexcursioncase}
As $n\to\infty$ through $2\NN$, 
conditional on $\tau>n$ and $S_n=(0,0)$, 
$$(N^1_n, N^2_n) \Rightarrow (X_1,X_2),$$
where $X_1$ and $X_2$ are independent with a $\BN(2,1/2)$ distribution.
\end{thm}

\begin{rem}
    In Theorems~\ref{thm:2dbridgecase} and \ref{thm:2dexcursioncase}, the condition $n\in 2\mathbb{N}$ simply results from the periodicity of the random walk together with the conditioning on $S_n=(0,0)$. The situation is completely different in the meander case of Theorem~\ref{thm:2dmeandercase}. In this case, as will be shown in the proof section, we obtain a different convergence in distribution when $n\to\infty$ along $2\mathbb{N}+1$. This parity phenomenon is a consequence of the periodicity of the underlying coordinate random walks on $\mathbb{Z}$.
\end{rem}

The structure of the proofs is as follows. A key guiding principle is our aim to keep the arguments as elementary as possible.
Although the coordinates of a simple two-dimensional random walk are not independent, conditioning on the number of horizontal (or vertical) steps restores a form of independence. In combinatorial terms, this corresponds to shuffling two one-dimensional random walks. This decomposition is compatible with several observables, such as the number of returns to the coordinate axes.
We can then express the probability that the number of returns of the two-dimensional walk takes a given value as a binomial sum of products of one-dimensional probabilities. In the next step, we invoke certain ad hoc Abelian theorems,
%Bernstein-type functions (which appear in the classical Weierstrass theorem on the density of polynomials in continuous functions on a segment), 
allowing us to transfer one-dimensional estimates to the two-dimensional setting. Section~\ref{sec:ingredients} illustrates these ideas through the example of the first exit time from the quarter plane.

The above discussion shows that the key estimates are one-dimensional, which is the purpose of Section~\ref{sec:dim1}. In particular, Theorem~\ref{thm:convergencetohalfnormal} (resp.\ \ref{thm:bridgeconvergencecase}, \ref{thm:excursionconvergencecase}, and \ref{thm:meanderconvergencecase}) is the one-dimensional analogue of Theorem~\ref{thm:2dunconditioned} (resp.\ \ref{thm:2dbridgecase}, \ref{thm:2dmeandercase}, and \ref{thm:2dexcursioncase}). Most of these one-dimensional results are classical (see, for instance, Feller \cite{Fel68,Fel71}, together with more recent combinatorial works \cite{BaFl-02,FlaSed09,BaWa-14,Wa-16,Wa-20,BaKuWaWa-24,BaKuWa-24})\ but for completeness and notational convenience we include self-contained proofs. To the best of our knowledge, the second item (negative drift, meander case) of Theorem~\ref{thm:meanderconvergencecase} demonstrates a new phenomenon: in this case, the limiting distribution of the number of returns to~$0$ is a mixture of two laws (namely, geometric and negative binomial).  Section~\ref{sec:proofs} contains the main proofs of our two-dimensional results. Let us mention in particular the proof of Theorem~\ref{thm:2dmeandercase}, which shows concrete cases where the coordinates of the limiting distribution are not independent. Appendix~\ref{sec:asymptoticequiofsums} gathers a collection of simple facts on asymptotic equivalence.

Let us conclude the introduction with two remarks on the above theorems. First, Theorems~\ref{thm:2dunconditioned}, \ref{thm:2dbridgecase}, \ref{thm:2dmeandercase}, and \ref{thm:2dexcursioncase} can be extended straightforwardly to any dimension $d$ (not only $d=2$). We chose to restrict ourselves to dimension $2$ in this paper because the general $d$-dimensional case would lead to heavier notation without introducing any new ideas. The statements would be identical, with the number of coordinates simply increased, and the limiting distribution $(X_1,\ldots,X_d)$ involving independent conditions on each coordinate.

In another direction, and with some additional effort, our results could be extended to certain generalized simple random walks in dimension $d$, under the assumption that only one coordinate moves at a time (though not necessarily to a nearest neighbor). In this setting, one essentially lifts one-dimensional results to higher dimensions: as soon as precise estimates are available in dimension $1$, our methods apply. We do not pursue this extension here and instead refer the reader to the existing literature on one-dimensional walks.

\section{A key tool: the binomial convolution}
\label{sec:ingredients}

\subsection{Path decomposition and independence}
\label{sec:pathdecompandindependence}

Starting from the $2$-dimensional random walk $\left(S^1_n, S^2_n\right)_{n\geq 0}$, we extract two simple random walks as follows. 
For $i=1,2$, consider the sequence of stopping times defined by $T^i_0=0$ and 
\begin{equation*}
T^i_{n+1}=\inf\{ k> T^i_n \mid  S^i_k\not= S^i_{k-1} \},
\end{equation*}
and set $\tS^i_n=S^i_{T^i_n}$.
The sequence $\bigl(\tS^i_n\bigr)_{n\geq 0}$ is a simple random walk on $\ZZ$ with transition probabilities $\widetilde{p}_i=p_i/h_i$ for a step to the right and $\widetilde{q}_i=q_i/h_i$ for a step to the left. See Figure~\ref{fig:RW_dim2_tilde}. We also consider the sequence $\left(D_n\right)_{n\geq 1}$ defined by $D_n=\mathds{1}_{\{S^1_n\not=S^1_{n-1}\}}$ that indicates whether the $n$-th step of the two-dimensional random walk is horizontal or vertical. 

\begin{figure}[ht!]
\begin{center}
  \begin{tikzpicture}[
    x=5.5mm,
    y=5.5mm,
  ]

    % Grid
     \draw (-3,-3) rectangle (3,-1);
      \draw  (-3,-3) rectangle (-1,3);
    \draw[thin]
      \foreach \x in {-3, ..., 3} {
        (\x, -3) -- (\x, 3)
      }
      \foreach \y in {-3, ..., 3} {
        (-3, \y) -- (3, \y)
      }
    ;  

    % Coordinate axes
    \begin{scope}[
      thick,
      ->,
      >={Stealth[]},
    ]
   
      \draw (-1, -3.5) -- (-1, 3.5);
      \draw (-3.5,-1) -- (3.5, -1);
    \end{scope}
  \draw[ultra thick,red] (-1,-1) -- (0,-1);
  \draw[very thick,blue] (0,-1) -- (0,0);
\draw[very thick,red] (0,0) -- (-1,0) -- (-2,0);
\draw[very thick,blue] (-2,0) -- (-2,1);
\draw[very thick,red] (-2,1) -- (-1,1);
\draw[very thick,blue] (-1,1) -- (-1,2);
\draw[very thick,red] (-1,2) -- (0,2) -- (1,2);
  \draw[very thick,blue] (1,2) -- (1,1);
\draw[very thick,red] (1,1) -- (0,1);
  \draw[very thick,blue] (0,1) -- (0,2) -- (0,3);
  
\filldraw[black] (-1,-1) circle (1.5pt);
\filldraw[black] (0,-1) circle (1.5pt);
\filldraw[black] (0,0) circle (1.5pt);
\filldraw[black] (-1,0) circle (1.5pt);
\filldraw[black] (-2,0) circle (1.5pt);
\filldraw[black] (-2,1) circle (1.5pt);
\filldraw[black] (-1,1) circle (1.5pt);
\filldraw[black] (-1,2) circle (1.5pt);
\filldraw[black] (0,2) circle (1.5pt);
\filldraw[black] (1,2) circle (1.5pt);
\filldraw[black] (1,1) circle (1.5pt);
\filldraw[black] (0,1) circle (1.5pt);
\filldraw[black] (0,2) circle (1.5pt);
\filldraw[black] (0,3) circle (1.5pt);

    \end{tikzpicture}\quad
   \begin{tikzpicture}[
    x=5.5mm,
    y=5.5mm,
  ]

    % Grid
     \draw (-3,-3) rectangle (4,-1);
      \draw  (-3,-3) rectangle (-1,3);
    \draw[thin]
      \foreach \x in {-3, ..., 4} {
        (\x, -3) -- (\x, 3)
      }
      \foreach \y in {-3, ..., 3} {
        (-3, \y) -- (4, \y)
      }
    ;  

    % Coordinate axes
    \begin{scope}[
      thick,
      ->,
      >={Stealth[]},
    ]
   
      \draw (-3, -3.5) -- (-3, 3.5);
      \draw (-3.5,-1) -- (4.5, -1);
    \end{scope}
  \draw[ultra thick,red] (-3,-1) -- (-2,0) -- (-1,-1) -- (0,-2) -- (1,-1) -- (2,0) -- (3,1) -- (4,0);
  
\filldraw[black] (-3,-1) circle (1.5pt);
\filldraw[black] (-2,0) circle (1.5pt);
\filldraw[black] (-1,-1) circle (1.5pt);
\filldraw[black] (0,-2) circle (1.5pt);
\filldraw[black] (1,-1) circle (1.5pt);
\filldraw[black] (2,0) circle (1.5pt);
\filldraw[black] (3,1) circle (1.5pt);
\filldraw[black] (4,0) circle (1.5pt);

    \end{tikzpicture}\quad
     \begin{tikzpicture}[
    x=5.5mm,
    y=5.5mm,
  ]

    % Grid
     \draw (-3,-3) rectangle (3,-1);
      \draw  (-3,-3) rectangle (-1,3);
    \draw[thin]
      \foreach \x in {-3, ..., 3} {
        (\x, -3) -- (\x, 3)
      }
      \foreach \y in {-3, ..., 3} {
        (-3, \y) -- (3, \y)
      }
    ;  

    % Coordinate axes
    \begin{scope}[
      thick,
      ->,
      >={Stealth[]},
    ]
   
      \draw (-3, -3.5) -- (-3, 3.5);
      \draw (-3.5,-1) -- (3.5, -1);
    \end{scope}
  \draw[ultra thick,blue] (-3,-1) -- (-2,0) -- (-1,1) -- (0,2) -- (1,1) -- (2,2) -- (3,3);
  
\filldraw[black] (-3,-1) circle (1.5pt);
\filldraw[black] (-2,0) circle (1.5pt);
\filldraw[black] (-1,1) circle (1.5pt);
\filldraw[black] (0,2) circle (1.5pt);
\filldraw[black] (1,1) circle (1.5pt);
\filldraw[black] (2,2) circle (1.5pt);
\filldraw[black] (3,3) circle (1.5pt);

    \end{tikzpicture}
    
    \caption{Left: representation of the random walk $\bigl(S_n^1,S_n^2\bigr)_{n\geq 0}$. Middle: graph of $\bigl({\widetilde S}_n^1\bigr)_{n\geq 0}$. Right: graph of $\bigl({\widetilde S}_n^2\bigr)_{n\geq 0}$}
    \label{fig:RW_dim2_tilde}
\end{center}    
\end{figure}

It is clear that $\left(D_n\right)_{n\geq 1}$ is an i.i.d.\ sequence of Bernoulli trials with probability of success equal to $h_1$. The random variable $H_n=D_1+D_2+\cdots+D_n$ that counts the number of horizontal steps after $n$ steps has a binomial $\B(n,h_1)$ distribution. The elementary but nevertheless key preliminary result below shows that the above variables satisfy the following independence property:

\begin{lem}
\label{lem:shuffleindependence}
For all $0\leq k\leq n$, 
the sequences $(\tS^1_1, \tS^1_2, \ldots, \tS^1_k)$ and $(\tS^2_1, \tS^2_2, \ldots, \tS^2_{n-k})$ and the event $\{H_n=k\}$ are independent.
\end{lem}
\begin{proof}
Write $\txi^i_n=\tS^i_n-\tS^i_{n-1}$ for the steps of the random walks. The statement of the lemma is equivalent to the independence of the sequences $(\txi^1_1, \txi^1_2, \ldots, \txi^1_k)$, $(\txi^2_1, \txi^2_2, \ldots, \txi^2_{n-k})$ and the event $\{H_n=k\}$.

Let $(\epsilon_1, \epsilon_2, \ldots, \epsilon_k)\in\{-1,1\}^k$, $(\eta_1, \eta_2, \ldots, \eta_{n-k})\in\{-1,1\}^{n-k}$, and $(\gamma_1,\gamma_2,\ldots,\gamma_n)\in\{0,1\}^n$ such that $\sum_{i=1}^n \gamma_i=k$. The conditions
$$\txi^1_1=\epsilon_1,\ldots, \txi^1_k=\epsilon_k, \txi^2_1=\eta_1,\ldots, \txi^2_{n-k}=\eta_{n-k}, D_1=\gamma_1, \ldots, D_n=\gamma_n$$
describe a unique trajectory of the original $2$-dimensional random walk that is obtained by assigning, in that precise order, the values $(\epsilon_i,0)$, $i=1,\ldots, k$ at the steps $\ell$ such that $D_{\ell}=1$, and the values $(0,\eta_j)$, $j=1,\ldots, n-k$ at the steps $\ell$ where $D_{\ell}=0$. Such a trajectory has a probability equal to
$$p=p_1^{I}q_1^{k-I}p_2^{J}q_2^{(n-k)-J}$$
where $I=\sharp\{1\leq i \leq k \mid \epsilon_i=1\}$ is the number of steps to the East and $J=\sharp\{1\leq j \leq n-k \mid \eta_j=1\}$ the number of steps to the North. Using the relations $p_i=\widetilde{p_i}h_i$ and $q_i=\widetilde{q_i}h_i$, this probability can be written in the following way:
$$p=\widetilde{p}_1^{I}\widetilde{q}_1^{k-I}\widetilde{p}_2^{J}\widetilde{q}_2^{(n-k)-J}h_1^kh_2^{n-k}.$$
Summing over all tuples $(\gamma_1,\gamma_2,\ldots,\gamma_n)\in\{0,1\}^n$ such that $\sum_{i=1}^n \gamma_i=k$ gives
\begin{multline*}
\PP(\txi^1_1=\epsilon_1,\ldots, \txi^1_k=\epsilon_k, \txi^2_1=\eta_1,\ldots, \txi^2_{n-k}=\eta_{n-k}, H_n=k) \\
 =\widetilde{p}_1^{I}\widetilde{q}_1^{k-I}\widetilde{p}_2^{J}\widetilde{q}_2^{(n-k)-J}\binom{n}{k}h_1^kh_2^{n-k}.
\end{multline*}
This proves independence.
\end{proof}

\subsection{The method: An example}

This decomposition of the $2$-dimensional random walk enables us to derive asymptotic results from the $1$-dimensional ones. To illustrate the method, consider the exit time 
$$\tau=\inf\{ n \geq 1 \mid  S^1_n<0 \mbox{ or } S^2_n<0\}$$
of the $2$-dimensional random walk from the non-negative quadrant $\NN_0^2$. Let also $\widetilde{\tau}_i$ be the exit time of the random walk $\bigl(\tS^i_n\bigr)_{n\geq 0}$ from the half line $[0,\infty)$. Under the condition $H_n=k$ (meaning that exactly $k$ horizontal steps have occurred by time $n$)\ the assertion $\tau>n$ is equivalent to $(\widetilde{\tau}_1>k, \widetilde{\tau}_2>n-k)$. Therefore
\begin{align}
\label{eq:exittimedecomp2d}
\PP(\tau>n) & = \sum_{k=0}^n\PP(\widetilde{\tau}_1>k, \widetilde{\tau}_2>n-k, H_n=k)\\
\nonumber    & = \sum_{k=0}^n\PP(\widetilde{\tau}_1>k)\PP(\widetilde{\tau}_2>n-k) \binom{n}{k} h_1^kh_2^{n-k}.
    \end{align}
Select $\alpha$ and $\beta$ such that 
$0<\alpha<h_1< \beta<1$.
The sum 
$$\sum_{k<\alpha n \mbox{\small\  or } k>\beta n}\binom{n}{k} h_1^k h_2^{n-k}$$
is the probability that a random variable with binomial distribution $\B(n,h_1)$ lies outside the interval $[\alpha n,\beta n]$, therefore it is an $O(\nu^n)$, for some $0<\nu<1$. This follows readily from classical Chernoff bounds; see Lemma~\ref{lem:binomial_tails} and its proof for a short argument. By consequence 
\begin{equation*}
\PP(\tau>n) = \Sigma_n + O(\nu^n),
\end{equation*}
where
$$\Sigma_n=\sum_{\alpha n\leq k\leq \beta n}\PP(\widetilde{\tau}_1>k)\PP(\widetilde{\tau}_2>n-k)  \binom{n}{k} h_1^kh_2^{n-k}.$$
Now, assume the walk is centered, that is $p_1=q_1$ and $p_2=q_2$. Then, the random walks $\bigl(\tS^i_n\bigr)_{n\geq 0}$ are simple symmetric random walks on $\ZZ$ and
$$\PP(\widetilde{\tau}_i>n)\sim \frac{\kappa}{\sqrt{n}},$$
as $n\to\infty$, where $\kappa=\sqrt{2/\pi}$. It easily follows that 
\begin{align*}
\Sigma_n & \sim \sum_{\alpha n\leq k\leq \beta n}\frac{\kappa^2}{\sqrt{k}\sqrt{n-k}}\binom{n}{k} h_1^kh_2^{n-k}\\
&\sim \frac{\kappa^2}{n}\sum_{\alpha n\leq k\leq \beta n}\left(\frac{k}{n}\right)^{-1/2}\left(1-\frac{k}{n}\right)^{-1/2}\binom{n}{k} h_1^kh_2^{n-k}.
\end{align*}
Lemma~\ref{lem:Abelianforbinomialtransform} from the next section ensures that the last sum converges to $\left(h_1h_2\right)^{-1/2}$, so that we obtain the estimate (which is only valid for centered random walks)
\begin{equation}
\PP(\tau>n)\sim \frac{\kappa^2}{n\sqrt{h_1h_2}}.
\end{equation}
Of course, this result is well known (see \cite{DeWa15} for example). We shall use the same general method to derive other asymptotics.

\subsection{Binomial convolution asymptotics}
\label{sec:binomialconvolution}

The following general asymptotic result will be used repeatedly. It stems from an adaptation of the classical proof of Weierstrass theorem via Bernstein polynomials.

\begin{lem}
\label{lem:Abelianforbinomialtransform}
Let $0\leq\alpha<h<\beta\leq 1$ and $a\in \NN$ and $b\in \NN_0$ be fixed. Let also $f:[\alpha,\beta]\to \RR$ be a continuous function. Then
$$\lim_{n\to\infty}\sum_{\substack{\alpha n\leq k\leq \beta n,\\ k\in a\NN+b}}f\left(\frac{k}{n}\right) \binom{n}{k}h^k(1-h)^{n-k}=\frac{f(h)}{a}.$$
\end{lem}
\begin{proof}
Define the oscillation of order $\delta$ of $f$ by
$$\omega(f, \delta)=\sup_{\substack{\vert x-y\vert <\delta,\\ x,y\in[\alpha, \beta] }}\vert f(x)-f(y)\vert.$$
Since $f$ is uniformly continuous, we have
\begin{equation}
\label{eq:unifequicontinuity}
\lim_{\delta\to 0}\omega(f, \delta)=0.
\end{equation}
Let $X_n$ be a binomial random variable with parameters $(n,h)$. The quantity we wish to estimate is 
$$F(n)=\EE(f(X_n/n), X_n\in [\alpha n, \beta n], X_n\in a\NN_0+b).$$
(Here and throughout, for a random variable $X$ and an event $A$, we write $\EE(X,A)$ the expectation $\EE(X\mathds{1}_A)$.)
Let $\delta>0$. First we have
\begin{align*}
\EE(\vert f(X_n/n)-f(h)\vert, X_n\in [\alpha n, \beta n]) 
    & \leq 2M\PP(\vert X_n/n-h \vert\geq \delta)+ \omega(f, \delta) \\
      & \leq 2M \frac{h(1-h)}{n\delta^2}+ \omega(f, \delta),
\end{align*}
where we have used Bienaym\'e-Chebychev on the second line, and 
$M=\sup_{x\in [\alpha,\beta]}\vert f(x)\vert$. Taking the $\limsup$ as $n\to\infty$ gives
$$
\limsup_{n\to\infty}\EE(\vert f(X_n/n)-f(h)\vert, X_n\in [\alpha n, \beta n]) \leq \omega(f, \delta).
$$
Since this inequality is true for all $\delta>0$, it follows from \eqref{eq:unifequicontinuity} that 
$$
\lim_{n\to\infty} \EE(\vert f(X_n/n)-f(h)\vert, X_n\in [\alpha n, \beta n]) =0.
$$
Of course, this remains true if we add the condition $X_n\in a\NN_0+b$ inside the expectation, so that
$$\lim_{n\to \infty} \left\vert F(n)- f(h)\PP(X_n\in [\alpha n, \beta n], X_n\in a\NN_0+b)\right\vert =0.$$
To conclude, we have to show that 
$$\lim_{n\to\infty }\PP(X_n\in [\alpha n, \beta n], X_n\in a\NN_0+b)=1/a.$$
This follows from the fact that $\lim_{n\to\infty }\PP(X_n<\alpha n \mbox{ or } X_n>\beta n)=0,$ and $\lim_{n\to\infty }\PP(X_n\in a\NN_0+b)=1/a$ (this last limit is classical; it can be derived from an exact formula for fixed $n$ obtained via series multisection, see \cite{BeDo-16} for example).
\end{proof}
Notice that the proof of Lemma~\ref{lem:Abelianforbinomialtransform} actually provides a slightly stronger result than the one stated, as it can be used to obtain uniform estimates with respect to the parameter $h$.

\section{Estimates and limit theorems for the simple random walk on
\texorpdfstring{$\ZZ$}{Z}}
\label{sec:dim1}
In this section, we gather a number of estimates and prove limit theorems for the simple random walk on $\ZZ$. Most of them are known, but some seem to be new (see for example Theorem~\ref{thm:meanderconvergencecase} \ref{it:new_case}). The results of this section will serve as a basis to obtain estimates and limit theorems in dimension~$2$.

 So, in this section, we consider a one-dimensional random walk $\left(S_n\right)_{n\geq 0}$ on $\ZZ$ started at $S_0=0$ and with i.i.d.\ increments $\xi_n=S_n-S_{n-1}$ satisfying
$\PP(\xi_k=1)=p$ and $\PP(\xi_k=-1)=q$, where $p,q$ are positive numbers such that $p+q=1$.

Let $N_n$ be the number of returns to $0$ until time $n$ defined by
$$N_n=\sharp\{1\leq k\leq n \mid  S_k=0\}.$$
We are interested in limit theorems for $N_n$ under various conditionings.

\subsection{Statements of one-dimensional limit theorems}
\label{sec:onedimstatements}

\begin{thm}[Unconditioned case]
\label{thm:convergencetohalfnormal}\hfill
\begin{enumerate}
\item If $p=q$, then for all $x\geq 0$,
$$\lim_{n\to\infty}\PP(N_n\leq x\sqrt{n} ) = \sqrt{\frac{2}{\pi}} \int_0^x e^{-\frac{t^2}{2}}\,dt.$$
In other words, $N_n/\sqrt{n}$ converges in distribution to a half-normal distribution.
\item If $p\not=q$, then $N_n$ converges a.s.~to a random variable with a geometric distribution with parameter $\vert p-q\vert$.
\end{enumerate}
\end{thm}
The above result can be found, for instance, in \cite[Theorem~4.2]{Wa-16}, in the slightly more general setting of Motzkin paths, where in addition to left and right moves the random walk may also stay at its current position with some probability. See also \cite[Theorem~3.6]{Wa-20}.

\begin{thm}[Bridge case]
\label{thm:bridgeconvergencecase}
For all $p\in (0,1)$, for all $x\geq 0$,
$$\lim_{n\to\infty}\PP(N_n\leq x\sqrt{n} \mid  S_n=0 ) = \int_0^x t e^{-\frac{t^2}{2}}\,dt = 1- e^{-\frac{x^2}{2}},$$
where the limit is taken through $2\NN$. In other words, conditional on $S_n=0$, $N_n/\sqrt{n}$ converges in distribution to a Rayleigh distribution.
\end{thm}

The above result can be found in \cite[Corollary~6.7]{BaKuWa-24}. We now introduce $\tau$, the first exit time from $[0,\infty)$, in order to state our next result. It is defined by 
$$\tau=\inf\{n \geq 0 \mid  S_n<0\}.$$

\begin{thm}[Non-negative bridge case]
\label{thm:excursionconvergencecase}
For all $p\in (0,1)$, for all $r\geq 1$, 
$$\lim_{n\to\infty} \PP(N_n=r \mid  \tau>n, S_n=0 ) = \frac{r}{2^{r+1}},$$
where the limit is taken through $2\NN$.
\end{thm}
The above result appears as Theorem~4.2 in \cite{BaWa-14}, in a slightly more general reflection-absorption framework. Contrary to Theorems~\ref{thm:convergencetohalfnormal}, \ref{thm:bridgeconvergencecase} and~\ref{thm:excursionconvergencecase}, Theorem~\ref{thm:meanderconvergencecase} below does not seem to appear in the existing literature, to the best of our knowledge.

\begin{thm}[Meander case]
\label{thm:meanderconvergencecase}\hfill
\begin{enumerate}[label=(\arabic{*}), ref=(\arabic{*})]
\item\label{it:0-new_case} If $p\geq q$, then, conditional on $\tau>n$, the random variable $N_n$ converges in distribution to a geometric distribution $\G(p)$, i.e., for all $r\geq 0$,
$$\lim_{n\to\infty}\PP(N_{n}= r \mid  \tau >n)=pq^r.$$
\item\label{it:new_case} If $p<q$, then for all $r\geq 0$,
$$\lim_{n\to\infty}\PP(N_{2n}= r \mid  \tau >2n)=\frac{2p+(1-2p)r}{2^{r+1}}$$
and
$$\lim_{n\to\infty}\PP(N_{2n+1}= r \mid  \tau >2n+1)=\frac{\frac{1}{2q}+\bigl(1-\frac{1}{2q}\bigr)r}{2^{r+1}}.$$
\end{enumerate}
\end{thm}

Note that, in Theorem~\ref{thm:meanderconvergencecase}, \ref{it:new_case}, the limiting distribution is a mixture of a geometric distribution and a negative binomial distribution. More precisely, it is the distribution of the random variable
\begin{equation*}
Z=\zeta X+ (1-\zeta) Y,
\end{equation*}
where $X, Y, \zeta$ are independent, $X$ is geometric $\G(1/2)$, $Y$ is negative binomial $\BN(2,1/2)$, and $\zeta$ is Bernoulli with parameter $2p$ or $\frac{1}{2q}$ according to the parity of $n$. Thus, the limiting distribution for a meander with negative drift interpolates between that of a zero-drift meander and that of a non-negative excursion.

In Theorem~\ref{thm:meanderconvergencecase} \ref{it:0-new_case}, the parameter of the geometric distribution interestingly differs from that appearing in the second item of Theorem~\ref{thm:convergencetohalfnormal}.
Note further that, in the bridge cases (Theorems~\ref{thm:bridgeconvergencecase} and~\ref{thm:excursionconvergencecase}), the limiting distribution does not depend on the drift $p-q$. This is due to the fact that, in those cases, the conditioning prescribes the position of the random walk at time $n$. This can be explained thanks to the classical Cram\'er's exponential change of measure; we refer to \cite{Cra-38} for the original contribution, and provide elementary considerations for the simple random walk below. Consider the Laplace transform of the increments distribution
$$L(t)=\EE(e^{t\xi_1})= pe^{t}+qe^{-t}.$$
It reaches its minimum at $t_0$ satisfying $e^{t_0}=\sqrt{\frac{q}{p}}$. At this point
$L'(t_0)=0$ and $L(t_0)=\sqrt{4pq}$.
If the increment distribution $\mu=p\delta_{1}+q\delta_{-1}$ is changed to 
$$\mu_*(dy)=\frac{e^{t_0 y}}{L(t_0)}\mu(dy),$$
then the random walk becomes symmetric. Indeed, if $\PP_*$ denotes a probability  measure under which the increments of the random walk $\left(S_n\right)_{n\geq 0}$ have distribution $\mu_*$, and $\EE_*$ is the associated expectation, then the new
Laplace transform  $L_*(t)=\EE_*(e^{t\xi_1})=\int e^{ty}\mu_*(dy)=L(t+t_0)/L(t_0)$ satisfies $L_*'(0)=0$, i.e., $\EE_*(\xi_1)=0$. Furthermore, it holds that
\begin{equation}
\label{eq:cramertransformation}
\EE(F(S_1, S_2, \ldots, S_n)) = L(t_0)^n  \EE_*(F(S_1, S_2, \ldots, S_n)e^{-{t_0}S_n}),  
\end{equation}
for any measurable function $F:\RR^n\to\RR$.
In particular, for any measurable subset $A\subset\RR^n$ and any $x\in \ZZ$, the following relation holds:
\begin{equation}
\label{eq:changeofmeasureforsimplerandomwalk}
\PP((S_1, S_2, \ldots , S_n)\in A, S_n=x)  = (\sqrt{4pq})^n \left(\sqrt{\frac{p}{q}}\right)^{x}\PP_*((S_1, S_2, \ldots, S_n)\in A, S_n=x).
\end{equation}
Consequently, 
if the walk is conditioned by any event that prescribes its value at time $n$, then the conditional probability is the same under $\PP$ or $\PP_*$, i.e., it does not depend on $p\in (0,1)$.

\subsection{One-dimensional estimates}

In order to analyse the number $N_n$, we shall use its connection with the stopping times $\theta_r$ which are defined inductively as follows:
$\theta_0=0$ and for $r\geq 0$,
$$\theta_{r+1}=\inf\{ n>\theta_{r} \mid  S_n=0\}.$$
The number $\theta_r$ is the time of the $r$-th return to $0$.

We collect in Lemma~\ref{lem:onedzerodriftbasicestimates} below a number of classical estimates from which we shall deduce the limit theorems.

\begin{rem}
Hereafter, the sentence ``uniformly for $r^2\leq Cn$'' means ``uniformly for $r^2\leq C n$, for any constant $C>0$''.
\end{rem}

\begin{lem}
\label{lem:onedzerodriftbasicestimates}
Let $\kappa=\sqrt{2/\pi}$. In case $p=q$, as $n\to\infty$, the following estimates hold:
\begin{alignat}{2}
\label{eq:1dthetarequaln}\PP(\theta_r=n) & \sim \frac{\kappa r}{n^{3/2}}e^{-\frac{r^2}{2n}},
 & \qquad&\mbox{for } n \mbox{ even}, \mbox{ uniformly for } r^2\leq Cn;\\
\label{eq:1dthetarequalnandtausupn}\PP(\theta_r=n, \tau>n) & \sim \frac{\kappa r}{2^r n^{3/2}}e^{-\frac{r^2}{2n}},
& &\mbox{for } n \mbox{ even}, \mbox{ uniformly for } r^2\leq Cn;\\
\label{eq:1dnumberofreturnisr}\PP(N_n=r) & \sim \frac{\kappa}{\sqrt{n}}e^{-\frac{r^2}{2n}},
& &\mbox{uniformly for } r^2\leq Cn;\\
\label{eq:1dtausupn} \PP(\tau >n) & \sim \frac{\kappa}{\sqrt{n}};  & &\\
\label{eq:1dzeroattimen}\PP(S_n=0) & \sim \frac{\kappa}{\sqrt{n}}, & &\mbox{for } n \mbox{ even};\\
\label{eq:1dzeroattimenandtausupn}\PP(\tau>n, S_n=0) & \sim  \frac{2\kappa}{n^{3/2}}, & &  \mbox{for } n \mbox{ even}.
\end{alignat}
\end{lem}
\begin{proof}
For $n$ even and $r\leq n/2$, we have
$$\PP(\theta_r=n)=\frac{1}{2^{n-r}} \frac{r}{n-r}\binom{n-r}{\frac{n}{2}},$$
see \cite[III.7, Theorem 4 p.90]{Fel68}. The asymptotic equivalence~\eqref{eq:1dthetarequaln} then follows from Stirling's formula (see also \cite[III.7, Eq.~(7.6) p.90]{Fel68}).

The relation~\eqref{eq:1dthetarequalnandtausupn} follows from the equality
\begin{equation}
\label{eq:flip}
\PP(\theta_r=n, \tau>n)=\frac{1}{2^r} \PP(\theta_r=n),
\end{equation}
which can be explained as follows. Any sample path $\omega$ of the simple random walk such that $\theta_r=n$ can be transformed into a path $\phi(\omega)$ having the same zeros and satisfying $\tau>n$ by flipping the negative parts of the path between two consecutive zeros. This transformation $\phi$ from $\{\theta_r=n\}$ to $\{\theta_r=n, \tau>n\}$ is onto and $\card\left(\phi^{-1}\{\omega'\}\right)=2^r$ for any image path $\omega'$ (since there are $r$ parts that can be flipped).

Concerning the number of returns to zero, it is known that
$$\PP(N_{2n}=r)= \frac{1}{2^{2n-r}}\binom{2n-r}{n},$$
see~\cite[III.10, Problems 9 \& 10 p.96]{Fel68}. Therefore,
$$\PP(N_{2n}=r)=\frac{2n-r}{r}\PP(\theta_r=2n)\sim \frac{\kappa}{\sqrt{2n}}e^{-\frac{r^2}{4n}},$$
uniformly for $r^2\leq Cn$. Since $N_{2n+1}=N_{2n}$, the same asymptotic equivalence holds for $\PP(N_{2n+1}=r)$ and \eqref{eq:1dnumberofreturnisr} follows easily.

It is well known that, for all $n\in 2\NN$,
$$\PP(S_1\not=0, S_2\not=0, \ldots, S_n\not=0)=\PP(S_n=0)=\frac{1}{2^{n}}\binom{n}{n/2},$$
see \cite[III.3, Lemma 1 p.76]{Fel68}.
Therefore,
\begin{align*}
\PP(S_{2n}=0) & = 2\PP(S_1>0, S_2>0,\ldots, S_{2n}>0)\\
 & =2\PP(S_1=1, S_2-S_1\geq 0,\ldots, S_{2n}-S_1\geq 0)\\
 & =2\PP(S_1=1)\PP(S_2-S_1\geq 0,\ldots, S_{2n}-S_1\geq 0)\\
 & =\PP(\tau>2n-1)\\
 & =\PP(\tau>2n).
\end{align*}
The asymptotics \eqref{eq:1dtausupn} and \eqref{eq:1dzeroattimen} follow from Stirling's formula.

Finally, we look at $\PP(\tau>n, S_n=0)$. Denote by $\tau=T_{-1}$ the hitting time of level $-1$. Since the walk is at $0$ the time before it hits $-1$, we have
\begin{equation*}
\PP(T_{-1}=n+1)  = \PP(S_1\geq 0, \ldots, S_{n-1}\geq 0,  S_n=0, S_{n+1}=-1)
  = \frac{1}{2} \PP(\tau>n, S_n=0).
\end{equation*}
Therefore, according to \cite[Theorem 2 p.89 and Eq.~(7.6)]{Fel68},
$$\PP(\tau>n, S_n=0)=2\PP(T_{-1}=n+1)\sim 2\sqrt{\frac{2}{\pi}}\frac{1}{n^{3/2}},$$
as $r$ is fixed and $n\to\infty$ through $2\NN$. This corresponds to \eqref{eq:1dzeroattimenandtausupn}.
\end{proof}

In order to handle the non-centered case $p\not=q$, we will need the additional estimate of Lemma~\ref{lem:onednonzerodriftbasicestimates} below.

\begin{lem}
\label{lem:onednonzerodriftbasicestimates}
Let $\kappa=\sqrt{2/\pi}$ and $\rho=\sqrt{4pq}$. In case $p<q$, as $n\to\infty$, 
\begin{equation}
\PP(\tau >n)  \sim c(n) \frac{\rho^n}{n^{3/2}},
\end{equation}
where $c(n)$ is $2$-periodic with
$$c(0)= \kappa \sqrt{\frac{q}{p}} \frac{\rho}{1-\rho^2}\quad\mbox{ and }\quad c(1)=\rho c(0).$$
\end{lem}
\begin{proof}
For the simple random walk, we denote by $\tau=T_{-1}$, the hitting time of level $-1$.
Therefore, according to~\eqref{eq:changeofmeasureforsimplerandomwalk} and \cite[Theorem 2 p.89 and Eq.~(7.6)]{Fel68}
$$\PP(\tau=n)=\PP(T_{-1}=n, S_n=-1)=\rho^n \sqrt{\frac{q}{p}} \PP_*(T_{-1}=n) \sim
\kappa\sqrt{\frac{q}{p}}\frac{\rho^n}{n^{3/2}}$$
as $n\to\infty$ through $2\NN+1$. Note also that $\PP(\tau=n)=0$ for $n$ even. Since $p<q$ the random walk drifts to $-\infty$, so that the probability that $\tau=\infty$ is $0$. Hence
$$
\PP(\tau>2n+1)=\sum_{k>n} \PP(\tau=2k+1) \sim \kappa \sqrt{\frac{q}{p}}  \sum_{k>n}\frac{\rho^{2k+1}}{(2k)^{3/2}},
$$
where the last equivalence holds because $\rho<1$. It can be seen by elementary arguments that, for any $a\in [0,1)$ and $\beta\in\RR$, 
$$\sum_{k>n} \frac{a^k}{k^\beta}\sim\frac{a^{n+1}}{(1-a) n^\beta}.$$
It follows that
$$
\PP(\tau>2n+1)\sim \kappa \sqrt{\frac{q}{p}} \frac{\rho^2}{1-\rho^2}\frac{\rho^{2n+1}}{(2n)^{3/2}}.
$$
Since $\PP(\tau>2n+2)=\PP(\tau>2n+1)$, we can write
$$\PP(\tau>2n+2)\sim \kappa \sqrt{\frac{q}{p}} \frac{\rho}{1-\rho^2}\frac{\rho^{2n+2}}{(2n)^{3/2}}.$$
This proves the lemma.
\end{proof}

\subsection{Proofs of one-dimensional limit theorems}

Let $\phi^+$ denote the cumulative distribution function of the half-normal distribution, defined by
$$\phi^+(x)=\kappa\int_0^x e^{-\frac{t^2}{2}}\,dt,$$
where $\kappa=\sqrt{2/\pi}$. Let also $\R$ denote the cumulative distribution function of the Rayleigh distribution:
$$\R(x)=\int_0^x te^{-\frac{t^2}{2}}\,dt=1-e^{-\frac{x^2}{2}}.$$
We begin with Lemma \ref{lem:1dunifasymptsfornumberreturn} below, which already contains Theorems \ref{thm:convergencetohalfnormal} and \ref{thm:bridgeconvergencecase}.
We present this result as a lemma here, since these uniform versions will be needed later to treat the two-dimensional case.
\begin{lem}
\label{lem:1dunifasymptsfornumberreturn}
For the simple random walk with $p=q$, the following asymptotics hold:
\begin{align}
\label{1dnumberreturnlessthanr}\PP(N_n< r) & \sim \phi^+\left(\frac{r}{\sqrt{n}}\right), \quad \mbox{ uniformly for } r^2\leq Cn;\\
\label{1dnumberreturnlessthanrandsnzero}\PP(N_n< r, S_n=0) & \sim 
\frac{\kappa}{\sqrt{n}}\R\left(\frac{r}{\sqrt{n}}\right), \quad \mbox{for }n \mbox{ even}, \mbox{ uniformly for } r^2\in [\epsilon n, Cn], \ \epsilon>0.
\end{align}
 \end{lem}
 \begin{proof}
 Since $\{N_{2n}<r\}=\{\theta_r>2n\}$ and $\theta_r$ is almost surely finite, we may write
 $\PP(N_{2n}<r)=\sum_{k>n}\PP(\theta_r=2k)$. Thus, it follows from the estimate \eqref{eq:1dthetarequaln} in Lemma~\ref{lem:onedzerodriftbasicestimates} together with Lemma~\ref{uniform_equivalence_of_remainders} of Section~\ref{sec:asymptoticequiofsums} that, as $n\to\infty$,
 $$\PP(N_{2n}<r)\sim \kappa \sum_{k>n} \frac{r}{(2k)^{3/2}}e^{-\frac{r^2}{4k}},$$
 uniformly for $r^2\leq Cn$. Now the last sum can be compared to an integral. Indeed, Lemma~\ref{uniform_equivalence_of_sums_and_integral} implies that, as $n\to\infty$,
$$\PP(N_{2n}<r)\sim  \kappa\int_{n}^{\infty} \frac{r}{(2s)^{3/2}}e^{-\frac{r^2}{4s}}\,ds,$$
 uniformly for $r^2\leq Cn$. Finally, the change of variable $s=\frac{r^2}{2t^2}$ shows that
 $$\int_{n}^{\infty} \frac{r}{(2s)^{3/2}}e^{-\frac{r^2}{4 s}}\,ds
 = \int_{0}^{r/\sqrt{2n}} e^{-\frac{t^2}{2}}\,dt.$$
 This establishes the asymptotics \eqref{1dnumberreturnlessthanr} for even $n$. The result then extends to all $n$ because $N_{2n+1} = N_{2n}$.

To obtain the asymptotics \eqref{1dnumberreturnlessthanrandsnzero}, first note that 
$$\PP(N_n=r, S_n=0)=\PP(\theta_r=n)\sim \frac{\kappa r}{n^{3/2}}e^{-\frac{r^2}{2n}},$$
for $n$ even, uniformly for $r^2\leq Cn$, see \eqref{eq:1dthetarequaln} in Lemma~\ref{lem:onedzerodriftbasicestimates}. Therefore,
$$\PP(N_n\leq r, S_n=0)=\sum_{k=1}^r \PP(N_n=k, S_n=0)\sim \kappa \sum_{k=1}^r 
\frac{k}{n^{3/2}}e^{-\frac{k^2}{2n}},$$
for $n$ even, uniformly for $r^2\leq Cn$.
Applying Lemma~\ref{lem:uniformriemannsum} to the function $f(t)=t e^{-\frac{t^2}{2}}$ with $\ell_n=\sqrt{n}$ gives
$$\lim_{n\to\infty} \frac{1}{\sqrt{n}}\sum_{k=1}^{\theta \sqrt{n}} \frac{k}{\sqrt{n}} e^{-\frac{k^2}{2n}}= \int_0^{\theta} t e^{-\frac{t^2}{2}}\,dt=\R(\theta),
$$
where the convergence holds uniformly for $\theta$ in any bounded interval $[0,T]$. In particular, for any $\epsilon$ such that $0<\epsilon<C$ the convergence is uniform on $[\sqrt{\epsilon}, \sqrt{C}]$. Since $\R(\theta)$ is bounded from below on this interval by the positive constant $\R(\sqrt{\epsilon})>0$, the following asymptotic equivalence holds uniformly for $\theta \in [\sqrt{\epsilon}, \sqrt{C}]$ as $n\to\infty$: 
$$\frac{1}{\sqrt{n}}\sum_{k=1}^{\theta \sqrt{n}} \frac{k}{\sqrt{n}} e^{-\frac{k^2}{2n}}\sim\R(\theta).
$$
The statement in \eqref{1dnumberreturnlessthanrandsnzero} is established by setting $r=\theta \sqrt{n}$.
\end{proof}

\begin{proof}[Proof of Theorem~\ref{thm:convergencetohalfnormal}]

In the case $p=q$, Lemma~\ref{lem:1dunifasymptsfornumberreturn} asserts that
$$\PP(N_n< r) \sim \phi^+\left(\frac{r}{\sqrt{n}}\right),$$
uniformly for  $r^2\leq Cn$. Taking $r=x\sqrt{n}$ gives
$$\PP(N_n< x\sqrt{n}) \sim \phi^+(x).$$
This proves the first assertion of Theorem~\ref{thm:convergencetohalfnormal}.

If $p\not= q$, then $\PP(\theta_1=\infty)=\vert p-q\vert >0$ (see  \cite[XIV.2, Eq.~(2.8)]{Fel68}, after applying the Markov property at time $1$). Let $N_{\infty}$ denote the total number of returns to zero. Due to the Markov property of the random walk,
$$\PP(N_{\infty}\geq r)=\PP(\theta_r<\infty)=\PP(\theta_1<\infty)^r,$$
therefore $N_{\infty}$ is finite a.s.~and has a geometric distribution on $\NN_0$ with parameter $\vert p-q\vert$. It follows that the non-decreasing sequence $(N_n)$ converges a.s.~to $N_\infty$.\end{proof}

\begin{proof}[Proof of Theorem~\ref{thm:bridgeconvergencecase}]

In the case $p=q$, Lemma~\ref{lem:1dunifasymptsfornumberreturn} asserts that
$$ \PP(N_n< r, S_n=0) \sim \frac{\kappa}{\sqrt{n}}\R\left(\frac{r}{\sqrt{n}}\right),$$
for $n$ even, uniformly for $r^2\in [\epsilon n, Cn]$, for all $0<\epsilon<C$. So it follows from the estimate~\eqref{eq:1dzeroattimen} in Lemma~\ref{lem:onedzerodriftbasicestimates} that
$$\PP(N_n< r \mid S_n=0) \sim \R\left(\frac{r}{\sqrt{n}}\right),$$
for $n$ even, uniformly for $r^2\in [\epsilon n, Cn]$, for all $0<\epsilon<C$. Taking $r=x\sqrt{n}$ proves the first assertion of Theorem~\ref{thm:bridgeconvergencecase}.

Now, a direct consequence of the change of measure~\eqref{eq:changeofmeasureforsimplerandomwalk} is that the probability $\PP(N_n< r \mid S_n=0)$ does not depend on the value of $p$. Therefore, the result is the same in the case $p\not=q$.\end{proof}

\begin{proof}[Proof of Theorem~\ref{thm:excursionconvergencecase}]
For any $n$ and $r$, we have
$$\PP(N_n=r, \tau>n, S_n=0 )=\PP(\theta_r=n, \tau>n).$$
So, it follows from \eqref{eq:1dthetarequalnandtausupn} in Lemma~\ref{lem:onedzerodriftbasicestimates}
that in the centered case $p=q$,
$$\PP(N_n=r, \tau>n, S_n=0 ) \sim \frac{\kappa r}{2^r n^{3/2}},$$
for $n$ even and all fixed $r\geq 1$. On the other hand, the same lemma asserts that
$$\PP(\tau>n, S_n=0)\sim \frac{2\kappa}{n^{3/2}},$$
as $n\to\infty$ through $2\NN$.
Therefore,
$$\lim_{n\to\infty}\PP(N_n=r \mid  \tau>n, S_n=0 )=\frac{r}{2^{r+1}},$$
where the limit is taken through $2\NN$.

Here again, a consequence of the change of measure~\eqref{eq:changeofmeasureforsimplerandomwalk} is that the probability $\PP(N_n= r \mid \tau>n, S_n=0)$ does not depend on the value of $p$. Therefore, the result is the same in the case $p\not=q$.
\end{proof}

\begin{proof}[Proof of Theorem~\ref{thm:meanderconvergencecase}]

First, we consider the case $p>q$, which is the easiest. For any fixed $r\geq 0$, the sequences of events $\left(\{\theta_r>n, \tau>n\}\right)_{n\geq 1}$ and $\left(\{\tau>n\}\right)_{n\geq 1}$ are non-increasing, thus
$$\lim_{n\to\infty}\PP(\theta_r>n, \tau>n) =\PP(\theta_r=\infty, \tau=\infty),$$
and
$$\lim_{n\to\infty}\PP(\tau>n) =\PP(\tau=\infty).$$
Since $p>q$, it follows that 
$\PP(\tau=\infty)=(p-q)/p>0$ (see \cite[XIV.2, Eq.~(2.8)]{Fel68}).
Therefore, combining the two preceding limits gives
$$\lim_{n\to\infty}\PP(N_n\geq r \mid \tau>n) =\lim_{n\to\infty}\PP(\theta_r\leq n \mid \tau>n) =\PP(\theta_r<\infty\mid \tau=\infty).$$
Thus it remains to show that
\begin{equation}
\label{eq:geometricinthecasepgreaterthanq}
\PP(\theta_r<\infty\mid \tau=\infty)=q^r.
\end{equation}
For $r=1$, we have
$$\PP(\theta_1<\infty\mid \tau=\infty)=1-\frac{\PP(\tau'=\infty)}{\PP(\tau=\infty)},$$
where $\tau'=\inf\{n\geq 1 \mid  S_n \leq 0\}$. Since
$\PP(\tau'>n)=p \PP(\tau>n-1)$ for all $n$,  we obtain
$$\PP(\theta_1<\infty\mid \tau=\infty)=1-p=q.$$
Then, for $r\geq 1$, the strong Markov property at time $\theta_r$ gives
$$
\PP(\theta_{r+1}<\infty, \tau=\infty) = \PP(\theta_r<\infty, \tau>\theta_r)\PP(\theta_1<\infty, \tau=\infty),$$
and 
$$
\PP(\theta_{r}<\infty, \tau=\infty) = \PP(\theta_r<\infty, \tau>\theta_r)\PP(\tau=\infty).
$$
From this we deduce
$$
\PP(\theta_{r+1}<\infty \mid  \tau=\infty) = \PP(\theta_{r}<\infty \mid  \tau=\infty) \PP(\theta_1<\infty \mid  \tau=\infty).
$$
This proves \eqref{eq:geometricinthecasepgreaterthanq}.

Now, we consider the case $p=q$. For any fixed $r\geq 0$, write
$$\PP(N_n=r, \tau>n)  =\PP(N_n=r, \tau >n , S_n\not=0)+\PP(N_n=r, \tau >n , S_n=0).$$
The same flipping argument as in the proof of \eqref{eq:1dthetarequalnandtausupn} in Lemma~\ref{lem:onedzerodriftbasicestimates} shows that
$$\PP(N_n=r, \tau >n , S_n\not=0)=\frac{1}{2^{r+1}} \PP(N_n=r, S_n\not=0),$$
whereas
$$\PP(N_n=r, \tau >n , S_n=0)=\frac{1}{2^r}\PP(N_n=r,  S_n=0).$$
(There are $r+1$ parts of the sample path that can be flipped in the first case, but only $r$ in the second.)
Therefore
\begin{equation*}
\PP(N_n=r, \tau>n)  =\frac{\PP(N_n=r)+\PP(N_n=r, S_n=0)}{2^{r+1}}
     =\frac{\PP(N_n=r)+\PP(\theta_r=n )}{2^{r+1}}.
\end{equation*}
We know from \eqref{eq:1dthetarequaln} and \eqref{eq:1dnumberofreturnisr} in Lemma~\ref{lem:onedzerodriftbasicestimates} that
$$\PP(N_n=r)\sim\frac{\kappa}{\sqrt{n}}\quad\mbox{ and }\quad \PP(\theta_r=n)= o(n^{-1/2}),$$
therefore,
\begin{equation}
\label{eq:1dnumberreturnrandtausupn}
\PP(N_n=r, \tau>n) \sim \frac{\kappa}{2^{r+1}\sqrt{n}}.
\end{equation}
By \eqref{eq:1dtausupn} in Lemma~\ref{lem:onedzerodriftbasicestimates}, we finally get
$$\lim_{n\to\infty} \PP(N_n=r \mid  \tau>n ) = \frac{1}{2^{r+1}}.$$

Finally, we consider the case $p<q$. By the Markov property of the random walk,
\begin{align*}
\PP(\theta_r\leq 2n, \tau >2n) & =\sum_{k=0}^n \PP(\theta_r=2k, \tau>2n)\\
 & = \sum_{k=0}^n \PP(\theta_r=2k, \tau>2k)\PP(\tau>2n-2k)\\
 & = \rho^{2n} \sum_{k=0}^n a_kb_{n-k},
\end{align*}
where
$a_n=\rho^{-2n}\PP(\theta_r=2n, \tau>2n)$
and
$b_n=\rho^{-2n}\PP(\tau>2n)$, with $\rho=\sqrt{4pq}$.
We know from relation~\eqref{eq:changeofmeasureforsimplerandomwalk} on page~\pageref{eq:changeofmeasureforsimplerandomwalk} and \eqref{eq:1dthetarequalnandtausupn} in Lemma~\ref{lem:onedzerodriftbasicestimates} that
$$a_n= \PP_*(\theta_r=2n, \tau>2n)\sim \frac{\kappa r}{2^r(2n)^{3/2}},$$
where $\PP_*$ is the probability measure~\eqref{eq:changeofmeasureforsimplerandomwalk} under which the random walk is centered.
We also know from Lemma \ref{lem:onednonzerodriftbasicestimates} that
$$b_n\sim  \frac{\rho}{1-\rho^2} \sqrt{\frac{q}{p}} \frac{\kappa}{(2n)^{3/2}}.$$
Therefore, it follows from Lemma \ref{lem:produitdeconvolution} below that
$$
\PP(\theta_r\leq 2n, \tau >2n) \sim \frac{\rho^{2n}}{n^{3/2}}(Ab+aB),
$$
where
$$a=\frac{\kappa r}{2^r 2^{3/2}},\qquad b=\frac{\rho}{1-\rho^2} \sqrt{\frac{q}{p}} \frac{\kappa}{2^{3/2}},$$
 $$A=\sum_{n=0}^\infty a_n=\sum_{n=0}^\infty \PP_*(\theta_r=2n, \tau >2n)= \sum_{n=0}^\infty \frac{\PP_*(\theta_r=2n)}{2^r}=\frac{\PP_*(\theta_r<\infty)}{2^r}=\frac{1}{2^r},$$
where we have used \eqref{eq:flip} above, and
$$B=\sum_{n=0}^\infty b_n=\sum_{n=0}^\infty \rho^{-2n}\PP( \tau >2n).$$
Since $\PP(\tau >2n)=\rho^{2n} b_n\sim b\rho^{2n}/n^{3/2}$, we obtain
$$\lim_{n\to \infty}\PP(\theta_r\leq 2n \mid  \tau >2n) = A+\frac{a}{b}B = \frac{1+cr}{2^r},$$
where 
\begin{equation*}
 c=\frac{1-\rho^2}{\rho}\sqrt{\frac{p}{q}} B.
\end{equation*}
The value of $c$ will be determined at the end of the proof by a simple argument.

Regarding the number of zeros, since $N_{2n}\geq r$ is equivalent to
$\theta_r\leq 2n$, this implies that for all $r\geq 1$,
\begin{equation*}
\lim_{n\to\infty}\PP(N_{2n}\geq r \mid  \tau >2n)=\frac{1+cr}{2^r}.
\end{equation*}
For $r=0$ the limit above is clearly $1$, so the formula is also true.
Equivalently, for all $r\geq 0$,
\begin{equation*}
\lim_{n\to\infty}\PP(N_{2n}= r \mid  \tau >2n)=\frac{1-c+cr}{2^{r+1}}.
\end{equation*}

In order to determine the constants involved, set $r=0$. We have then
$$\frac{1-c}{2}=\lim_{n\to\infty}\PP(N_{2n}= 0 \mid  \tau >2n)=\lim_{n\to\infty}\frac{\PP(\tau'>2n)}{\PP(\tau >2n)},$$
where $\tau'=\inf\{n\geq 1 \mid  S_n \leq 0\}$. Since
$$\PP(\tau'>2n)=p \PP(\tau>2n-1)= p \PP(\tau>2n),$$
we get 
$\frac{1-c}{2}= p$.
Therefore,
\begin{equation*}
\lim_{n\to\infty}\PP(N_{2n}= r \mid  \tau >2n)=\frac{2p+(q-p)r}{2^{r+1}}.
\end{equation*}
The case in which $2n$ is replaced by $2n+1$ gives rise to different constants; see Remark~\ref{rem:oddcase} below.
\end{proof}

\begin{rem}
\label{rem:oddcase}
Following the same method, we have
$$
\PP(\theta_r\leq 2n+1, \tau >2n+1) = \rho^{2n+1} \sum_{k=0}^n a_kb'_{n-k},
$$
with $\rho=\sqrt{4pq}$, $b'_n=\rho^{-(2n+1)}\PP(\tau>2n+1)\sim b'/n^{3/2}$, and $b'=\rho b$. It follows that
$$\PP(\theta_r\leq 2n+1, \tau >2n+1)\sim \frac{\rho^{2n+1}}{n^{3/2}}(Ab'+aB'),$$
with $A$, $a$ and $b'$ as above, and
$$B'=\sum_{n=0}^{\infty}\rho^{-(2n+1)}\PP(\tau>2n+1).$$
Therefore
$$\lim_{n\to \infty}\PP(\theta_r\leq 2n+1 \mid  \tau >2n+1) = A+\frac{a}{b'}B' = \frac{1+c'r}{2^r},$$
so, in the end, we obtain the same form of limit as in the even case, namely
$$\lim_{n\to \infty}\PP(N_{2n+1}=r \mid  \tau >2n+1) = \frac{(1-c')+c'r}{2^{r+1}}.$$
This time, however, using Lemma~\ref{lem:onednonzerodriftbasicestimates}, we obtain
\begin{equation*}
\PP(N_{2n+1}=0 \mid  \tau >2n+1)  =\frac{\PP(\tau'>2n+1)}{\PP(\tau>2n+1)}
 =p \frac{\PP(\tau>2n)}{\PP(\tau>2n+1)}
\sim \frac{p}{\rho^2} = \frac{1}{4q}.
\end{equation*}
Therefore, we see that $c' = 1 - \frac{1}{2q}$. Hence
$$\lim_{n\to \infty}\PP(N_{2n+1}=r \mid  \tau >2n+1) = \frac{\frac{1}{2q}+\bigl(1-\frac{1}{2q}\bigr)r}{2^{r+1}}.$$
\end{rem}

\section{Proofs of Theorems~\ref{thm:2dunconditioned}, \ref{thm:2dbridgecase}, \ref{thm:2dmeandercase} and \ref{thm:2dexcursioncase}}
\label{sec:proofs}

\subsection{Proof of Theorem~\ref{thm:2dunconditioned}}
\begin{proof}
Consider the numbers $N^i_n$ as defined in \eqref{eq:nindef}. Recall also the notation $H_n$ for the random number of horizontal steps. 
Since $N^i_n$ only records the ``true'' returns to zero, that is the number of indices $\ell$ such that $S^i_\ell=0$ but $S^i_{\ell-1}\not=0$, it is equal, under the condition $H_n=k$, to the number $\widetilde{N}^i_k$ of zeros of the random walk $(\tS^i_\ell)_{0\leq\ell\leq k}$; see the beginning of Section~\ref{sec:ingredients}, where we have introduced this notation. Therefore, according to Lemma~\ref{lem:shuffleindependence}, we have
\begin{align*}
\PP(N_n^1<r_1, N_n^2< r_2) & = \sum_{k=0}^n\PP(\widetilde{N}^1_k<r_1, \widetilde{N}^2_{n-k}<r_2, H_n=k)\\
\nonumber    & = \sum_{k=0}^n\PP(\widetilde{N}^1_k<r_1)\PP(\widetilde{N}^2_{n-k}<r_2) \binom{n}{k}h_1^k h_2^{n-k}.
    \end{align*}
Now, since $0<h_1<1$, we can select $\alpha$ and $\beta$ such that 
$0<\alpha<h_1< \beta<1$.
The sum 
$$\sum_{k<\alpha n \mbox{\small\  or } k>\beta n}\binom{n}{k} h_1^k h_2^{n-k}$$
is the probability that a random variable with binomial distribution $\B(n,h_1)$ lies outside the interval $[\alpha n,\beta n]$, therefore it follows from Lemma~\ref{lem:binomial_tails} that there exists $\nu \in (0,1)$ such that
\begin{equation}
\label{eq:remainderconroltheorem1}
\PP(N_n^1<r_1, N_n^2< r_2)  = F(n,r_1, r_2) + O(\nu^n),
\end{equation}
where
$$F(n,r_1, r_2)=\sum_{\alpha n\leq k\leq \beta n}\PP(\widetilde{N}^1_k<r_1)\PP(\widetilde{N}^2_{n-k}<r_2) \binom{n}{k} h_1^kh_2^{n-k}.$$ 

\emph{Assume $p_1=q_1$ and $p_2=q_2$}. In this case, the two random walks $(\tS^i_n)$ are simple symmetric random walks and, according to Lemma~\ref{lem:1dunifasymptsfornumberreturn}, as $n\to\infty$,
$$\PP(\widetilde{N}^i_n<r_i)\sim \phi_+\left(\frac{r_i}{\sqrt{n}}\right),$$
uniformly for $r_i^2\leq Cn$.
Therefore, setting 
$$
G(n, r_1, r_2)=\sum_{\alpha n\leq k\leq \beta n} \phi_+\left(\frac{r_1}{\sqrt{k}}\right) \phi_+\left(\frac{r_2}{\sqrt{n-k}}\right)\binom{n}{k} h_1^kh_2^{n-k},
$$
we have that
\begin{equation}
\label{eq:premiereequivalence}
F(n,r_1, r_2) \sim G(n, r_1, r_2),
\end{equation}
uniformly for $r_1^2, r_2^2\leq Cn$. 
Now let $s_1, s_2 >0$ be given and observe that
$$
G(n,s_1\sqrt{n},s_2\sqrt{n}) =\sum_{\alpha n\leq k\leq \beta n} f\left(\frac{k}{n}\right)\binom{n}{k} h_1^kh_2^{n-k},$$
where 
$$f(x)=\phi_+\left(\frac{s_1}{\sqrt{x}}\right) \phi_+\left(\frac{s_2}{\sqrt{1-x}}\right).$$ 
Since the function $f$ satisfies the assumptions of Lemma \ref{lem:Abelianforbinomialtransform} in Section~\ref{sec:binomialconvolution}, it follows that
$$\lim_{n\to\infty}G(n,s_1\sqrt{n},s_2\sqrt{n}) = \phi_+\left(\frac{s_1}{\sqrt{h_1}}\right) \phi_+\left(\frac{s_2}{\sqrt{h_2}}\right).$$
In combination with \eqref{eq:remainderconroltheorem1} and \eqref{eq:premiereequivalence}, this implies
$$\lim_{n\to\infty}\PP(N_n^1<s_1\sqrt{n}, N_n^2< s_2\sqrt{n})= \phi_+\left(\frac{s_1}{\sqrt{h_1}}\right) \phi_+\left(\frac{s_2}{\sqrt{h_2}}\right).$$
This proves the theorem in the case $p_1=q_1$ and $p_2=q_2$.

Now \emph{assume $p_1=q_1$ and $p_2\not=q_2$}. We still have 
$$\PP(\widetilde{N}^1_n<r_1)\sim \phi_+\left(\frac{r_1}{\sqrt{n}}\right),$$
uniformly for $r_1^2\leq Cn$ and, according to Theorem~\ref{thm:convergencetohalfnormal},
$$\lim_{n\to\infty}\PP(\widetilde{N}^2_n<r_2) = K(r_2)=1-\zeta^{r_2},$$
for all fixed $r_2$, where $\zeta=\vert p_2-q_2 \vert/h_2$. So, this time, setting
$$
G(n, r_1, r_ 2)=\sum_{\alpha n\leq k\leq \beta n} \phi_+\left(\frac{r_1}{\sqrt{k}}\right) K(r_ 2) \binom{n}{k} h_1^kh_2^{n-k},
$$
we have that
\begin{equation}
\label{eq:premiereequivalencebis}
F(n,r_1, r_2) \sim G(n, r_1, r_ 2),
\end{equation}
uniformly for $r_1^2 \leq Cn$ and any fixed $r_2$. 
Let $s_1>0$ be given. We have 
$$
G(n,s_1\sqrt{n},r_2) =K(r_ 2) \sum_{\alpha n\leq k\leq \beta n} f\left(\frac{k}{n}\right)\binom{n}{k} h_1^kh_2^{n-k},$$
where
$$f(x)=\phi_+\left(\frac{s_1}{\sqrt{x}}\right).$$ 
Since the function $f$ satisfies the assumptions of Lemma \ref{lem:Abelianforbinomialtransform} in Section~\ref{sec:binomialconvolution}, it follows that
$$\lim_{n\to\infty}G(n,s_1\sqrt{n},r_2) = \phi_+\left(\frac{s_1}{\sqrt{h_1}}\right) K(r_ 2).$$
In combination with \eqref{eq:remainderconroltheorem1} and \eqref{eq:premiereequivalencebis}, this implies
$$\lim_{n\to\infty}\PP(N_n^1<s_1\sqrt{n}, N_n^2< r_2)= \phi_+\left(\frac{s_1}{\sqrt{h_1}}\right) K(r_2).$$
This proves the theorem in the case $p_1=q_1$ and $p_2\not=q_2$.
The last case $p_1\not=q_1$ and $p_2\not=q_2$ can be treated in the same way.
\end{proof}

\subsection{Proof of Theorem~\ref{thm:2dbridgecase}}
\begin{proof}
The same decomposition argument as in the proof of Theorem \ref{thm:2dunconditioned} gives
\begin{align*}
\PP(N_n^1<r_1, N_n^2<& r_2, S_n=0) \\
 & = \sum_{k=0}^n\PP(\widetilde{N}^1_k<r_1, \tS^1_k=0, \widetilde{N}^2_{n-k}<r_2, \tS^2_{n-k}=0, H_n=k)\\
   & = \sum_{k=0}^n\PP(\widetilde{N}^1_k<r_1,\tS^1_k=0)\PP(\widetilde{N}^2_{n-k}<r_2, \tS^2_{n-k}=0) \binom{n}{k}h_1^k h_2^{n-k}.
    \end{align*}
Fix $\alpha$ and $\beta$ such that $0<\alpha<h_1< \beta<1$. It follows from Lemma~\ref{lem:binomial_tails} that there exists $\nu\in (0,1)$ such that
\begin{equation}
\label{eq:remainderconroltheorem2}
\PP(N_n^1<r_1, N_n^2< r_2, S_n=0)  = F(n,r_1, r_2) + O(\nu^n),
\end{equation}
where
$$F(n,r_1, r_2)=\sum_{\alpha n\leq k\leq \beta n}\PP(\widetilde{N}^1_k<r_1,\tS^1_k=0)\PP(\widetilde{N}^2_{n-k}<r_2, \tS^2_{n-k}=0) \binom{n}{k} h_1^kh_2^{n-k}.$$

\emph{Assume $p_1=q_1$ and $p_2=q_2$}. The two random walks $(\tS^i_n)$ are then simple symmetric random walks and, according to Lemma~\ref{lem:1dunifasymptsfornumberreturn}, as $n\to\infty$ through $2\NN$,
$$\PP(\widetilde{N}^i_n<r_i, \tS^i_n=0)\sim \frac{\kappa}{\sqrt{n}}\R\left(\frac{r_i}{\sqrt{n}}\right),$$
uniformly for $\epsilon n\leq r_i^2\leq Cn$, with $\kappa=\sqrt{\frac{2}{\pi}}$.
Therefore, setting
$$
G(n, r_1, r_ 2)=\kappa^2 \sum_{\substack{\alpha n\leq k\leq \beta n\\ k\in 2\NN}} \frac{1}{\sqrt{k(n-k)}}\R\left(\frac{r_1}{\sqrt{k}}\right) \R\left(\frac{r_2}{\sqrt{n-k}}\right)\binom{n}{k} h_1^kh_2^{n-k},
$$
we have that
\begin{equation}
\label{eq:premiereequivalence2}
F(n,r_1, r_2) \sim G(n, r_1, r_ 2),
\end{equation}
as $n\to\infty$ through $2\NN$, uniformly for $r_1^2, r_2^2\in[\epsilon n, Cn]$. 
Now let $s_1, s_2 >0$ be given and observe that
$$
G(n,s_1\sqrt{n},s_2\sqrt{n}) =\frac{\kappa^2}{n}\sum_{\alpha n\leq k\leq \beta n} f\left(\frac{k}{n}\right)\binom{n}{k} h_1^kh_2^{n-k},$$
where 
$$f(x)=\frac{1}{\sqrt{x(1-x)}}\R\left(\frac{s_1}{\sqrt{x}}\right) \R\left(\frac{s_2}{\sqrt{1-x}}\right).$$
It follows from Lemma \ref{lem:Abelianforbinomialtransform} that
$$G(n,s_1\sqrt{n},s_2\sqrt{n}) \sim \frac{\kappa^2}{2n\sqrt{h_1h_2}}\R\left(\frac{s_1}{\sqrt{h_1}}\right) \R\left(\frac{s_2}{\sqrt{h_2}}\right).$$
In combination with \eqref{eq:remainderconroltheorem2} and \eqref{eq:premiereequivalence2}, this implies
$$\PP(N_n^1<s_1\sqrt{n}, N_n^2< s_2\sqrt{n}, S_n=0)\sim \frac{\kappa^2}{2n\sqrt{h_1h_2}}\R\left(\frac{s_1}{\sqrt{h_1}}\right) \R\left(\frac{s_2}{\sqrt{h_2}}\right),$$
as $n\to\infty$ through $2\NN$.

Now, the same argument without the conditions $N^i_n<r_i$, together with the asymptotics $\PP(\tS^i_n=0)\sim \frac{\kappa}{\sqrt{n}}$, leads to
the estimate
$$\PP(S_n=0)\sim \frac{\kappa^2}{2n\sqrt{h_1h_2}},$$
as $n\to\infty$ through $2\NN$. Therefore,
$$\lim_{n\to\infty}\PP(N_n^1<s_1\sqrt{n}, N_n^2< s_2\sqrt{n} \mid  S_n=0)= \R\left(\frac{s_1}{\sqrt{h_1}}\right) \R\left(\frac{s_2}{\sqrt{h_2}}\right).$$
This proves the theorem in the case $p_1=q_1$ and $p_2=q_2$.

The convergence also holds in the general case since the probability under consideration does not depend on the parameters $p_1, p_2$. This follows by applying Cram\'er's change of measure argument, as explained at the end of Section~\ref{sec:onedimstatements}, see \eqref{eq:cramertransformation}. Here the bidimensional Laplace transform is given by
$$L(s,t)=\EE(e^{s\xi^1+t\xi^2})=p_1e^s+q_1e^{-s}+p_2e^{t}+q_2e^{-t}.$$
It reaches its minimum at the point $(s_0,t_0)$ such that
$e^{2s_0}=q_1/p_1$ and $e^{2t_0}=q_2/p_2$.\end{proof}

\subsection{Proof of Theorem~\ref{thm:2dmeandercase}}

\begin{proof}

We use again the decomposition argument of Section~\ref{sec:pathdecompandindependence}.
We recall that $\tau$ is the exit time of the two-dimensional random walk $\left(S_n\right)_{n\geq 0}$ from the quadrant $\mathbb N_0^2$. We shall first find a precise estimate of $\PP(\tau>n)$.

Let $\widetilde{\tau}_i$ denote the exit time of the random walk $\bigl(\tS^i_n\bigr)_{n\geq 0}$ from the half line $\mathbb N_0$.
 As usual, we write
\begin{align}
\label{eq:2dexittimedecomp}
\PP(\tau> n ) & = \sum_{k=0}^n\PP(\widetilde{\tau}_1>k, \widetilde{\tau}_2>n-k, H_n=k)\\
\nonumber   & = \sum_{k=0}^n\PP(\widetilde{\tau}_1>k)\PP(\widetilde{\tau}_2>n-k)\binom{n}{k}h_1^k h_2^{n-k}.
\end{align}
The asymptotic behaviour of $\PP(\widetilde{\tau}_i>n)$ depends on the drift $(p_i-q_i)/h_i$, and has the general form
\begin{equation*}
\PP(\widetilde{\tau}_i>n)\sim c_i(n)\frac{\rho_i^n}{n^{\alpha_i}},
\end{equation*}
where:
\begin{itemize}
\item In case $p_i>q_i$, one has $\rho_i=1$, $\alpha_i=0$ and $c_i(n)=\PP(\widetilde{\tau}_i=\infty)=(p_i-q_i)/p_i$ (see \cite[XIV.2, Eq.~(2.8)]{Fel68});
\item In case $p_i=q_i$, one has $\rho_i=1$, $\alpha_i=1/2$ and $c_i(n)=\sqrt{2/\pi}$ (see \eqref{eq:1dtausupn} in Lemma~\ref{lem:onedzerodriftbasicestimates});
\item In case $p_i<q_i$, one has $\rho_i=\sqrt{4p_iq_i}/h_i$, $\alpha_i=3/2$ and $c_i$ is $2$-periodic with (Lemma~\ref{lem:onednonzerodriftbasicestimates})
$$c_i(0)=\kappa \sqrt{\frac{q_i}{p_i}} \frac{\rho_i}{1-\rho_i^2}\quad\mbox{ and }\quad c_i(1)=\rho_i c_i(0).$$
\end{itemize}
In all three cases,
\begin{equation}
\label{eq:general_form_1d_exit_time}
g_i(n):=\rho_i^{-n}\PP(\widetilde{\tau}_i>n)\sim \frac{c_i(n)}{n^{\alpha_i}}
\end{equation}
is bounded by some constant $K_i\geq 0$. With this notation, we can rewrite \eqref{eq:2dexittimedecomp} as
\begin{equation*}
\PP(\tau> n ) = \theta^n \sum_{k=0}^n g_1(k)g_2(n-k)\binom{n}{k}\left(\frac{\rho_1h_1}{\theta}\right)^k \left(\frac{\rho_2h_2}{\theta}\right)^{n-k},
\end{equation*}
where $\theta=\rho_1h_1+\rho_2h_2$. Now, select $\alpha$ and $\beta$ such that $0<\alpha<\dfrac{\rho_1 h_1}{\theta}<\beta<1$ and write
\begin{equation*}
\PP(\tau> n ) = \theta^n \bigl( G(n)+ R(n)\bigr),
\end{equation*}
where
\begin{equation*}
G(n)=\sum_{\alpha n\leq k\leq \beta n} g_1(k)g_2(n-k)\binom{n}{k}\left(\frac{\rho_1h_1}{\theta}\right)^k \left(\frac{\rho_2h_2}{\theta}\right)^{n-k}.
\end{equation*}
Then,
$$\vert R(n) \vert \leq K_1K_2  \sum_{k<\alpha n \mbox{\small\  or } k>\beta n}\binom{n}{k} \left(\frac{\rho_1h_1}{\theta}\right)^k \left(\frac{\rho_2h_2}{\theta}\right)^{n-k},$$
and according to Lemma~\ref{lem:binomial_tails}, the last sum is $O(\nu^n)$ for some $\nu\in (0,1)$. Hence
$R(n)= O(\nu^n)$.
Considering $G(n)$, the estimate~\eqref{eq:general_form_1d_exit_time} leads to
\begin{align*}
G(n) & \sim \frac{1}{n^{\alpha_1+\alpha_2}}\sum_{\alpha n\leq k\leq \beta n}c_1(k)c_2(n-k)f\left(\frac{k}{n}\right)\binom{n}{k}\left(\frac{\rho_1h_1}{\theta}\right)^k \left(\frac{\rho_2h_2}{\theta}\right)^{n-k},
\end{align*}
 where 
$$f(x)=\frac{1}{x^{\alpha_1}(1-x)^{\alpha_2}}.$$
Now split the sum into two parts, according to $k\in 2\NN$ or $k\in2\NN+1$, and apply Lemma~\ref{lem:Abelianforbinomialtransform} to each part. For $n$ \emph{even}, this results in
$$
G(n)\sim \frac{1}{n^{\alpha_1+\alpha_2}} \left(\frac{c_1(0)c_2(0)+c_1(1)c_2(1)}{2}\right) f\left(\frac{\rho_1h_1}{\theta}\right).
$$
Since $R_n=O(\nu^n)$, we obtain
\begin{equation}
\label{eq:2destimatefortauinthenegativecase}
\PP(\tau>n)\sim  \frac{\theta^n}{n^{\alpha_1+\alpha_2}} \left(\frac{c_1(0)c_2(0)+c_1(1)c_2(1)}{2}\right) f\left(\frac{\rho_1h_1}{\theta}\right).\end{equation}
For $n$ \emph{odd}, a similar analysis leads to
$$\PP(\tau>n)\sim  \frac{\theta^n}{n^{\alpha_1+\alpha_2}} \left(\frac{c_1(0)c_2(1)+c_1(1)c_2(0)}{2}\right) f\left(\frac{\rho_1h_1}{\theta}\right).$$
Let us mention here the article \cite{Du-14}, which studies exit times and survival probabilities for random walks in cones with a certain negative drift. Although these results cannot be directly applied in our setting, due to the fact that our random walks are not aperiodic, we believe that they could be adapted to compute the exponential decay rate $\theta$ and the critical exponent $\alpha_1+\alpha_2$ in \eqref{eq:2destimatefortauinthenegativecase}.

In the same manner, we have
\begin{align}
\label{eq:remainderconroltheorem3}
\nonumber \PP(N_n^1&=r_1, N_n^2= r_2, \tau> n ) \\
 &=\sum_{k=0}^n\PP(\widetilde{N}^1_k=r_1, \widetilde{\tau}_1>k)\PP( \widetilde{N}^2_{n-k}=r_2, \widetilde{\tau}_2>n-k)\binom{n}{k}h_1^k h_2^{n-k}.
\end{align}
According to Theorem \ref{thm:meanderconvergencecase}, as $n\to\infty$,
$$\PP(\widetilde{N}^i_n=r_i, \widetilde{\tau}_i>n)\sim \phi_i(n,r_i) \PP(\widetilde{\tau}_i>n),$$
where
\begin{itemize}
\item In case $p_i\geq q_i$, one has 
$$\phi_i(n, r_i)=\frac{p_i}{h_i}\left(\frac{q_i}{h_i}\right)^{r_i};$$
\item In case $p_i<q_i$, $\phi_i(n, r_i)$ is 2-periodic w.r.t.~its first argument, and given by
$$\phi_i(0,r_i)=\frac{a_i+(1-a_i)r_i}{2^{r_i+1}}\quad\mbox{ and }\quad \phi_i(1,r_i)=\frac{b_i+(1-b_i)r_i}{2^{r_i+1}},$$
where $a_i=2p_i/h_i$ and $b_i=h_i/2q_i$. 
\end{itemize}
Combined with the estimate for $\PP(\widetilde{\tau}_i>n)$, this gives
$$\PP(\widetilde{N}^i_n=r_i, \widetilde{\tau}_i>n)\sim  \phi_i(n,r_i) c_i(n)\frac{\rho_i^n}{n^{\alpha_i}},$$
and the same arguments as for the estimate \eqref{eq:2destimatefortauinthenegativecase} of $\PP(\tau>n)$ lead to 
\begin{align*}
\nonumber \PP(N_n^1=r_1&, N_n^2=r_2, \tau>n)\sim\\
& \frac{\theta^n}{n^{\alpha_1+\alpha_2}} \left(\frac{c_1(0)\phi_1(0,r_1)c_2(0)\phi_2(0,r_2)+c_1(1)\phi_1(1,r_1)c_2(1)\phi_2(1,r_2)}{2}\right) f\left(\frac{\rho_1h_1}{\theta}\right),
\end{align*}
for $n$ even.
Consequently, as $n\to\infty$ through $2\NN$, using once again \eqref{eq:2destimatefortauinthenegativecase},
\begin{align}
\label{eq:horribleexpressionforthelimiteven}
\nonumber\lim \PP(N_n^1=r_1, N_n^2&=r_2 \mid  \tau>n) = \\&\frac{c_1(0)\phi_1(0,r_1)c_2(0)\phi_2(0,r_2)+c_1(1)\phi_1(1,r_1)c_2(1)\phi_2(1,r_2)}{c_1(0)c_2(0)+c_1(1)c_2(1)}.
\end{align}
It is easily seen that, for $n$ odd, the limit becomes
\begin{equation}
\label{eq:horribleexpressionforthelimitodd}
\frac{c_1(0)\phi_1(0,r_1)c_2(1)\phi_2(1,r_2)+c_1(1)\phi_1(1,r_1)c_2(0)\phi_2(0,r_2)}{c_1(0)c_2(1)+c_1(1)c_2(0)}.
\end{equation}
We now analyse the limits \eqref{eq:horribleexpressionforthelimiteven} and \eqref{eq:horribleexpressionforthelimitodd} for the different possible drifts.

\subsubsection{Case $p_1\geq q_1$ and $p_2\geq q_2$}
In this case, both $c_i(n)$ and $\phi_i(n, r_i)$ do not depend on $n$, therefore the limits \eqref{eq:horribleexpressionforthelimiteven} and \eqref{eq:horribleexpressionforthelimitodd} both simplify to
$$\phi_1(0, r_1)\phi_2(0, r_2)=\frac{p_1}{h_1}\left(\frac{q_1}{h_1}\right)^{r_1}\frac{p_2}{h_2}\left(\frac{q_2}{h_2}\right)^{r_2}.$$
This establishes convergence along $\NN$ and identifies the limiting distribution as a pair of independent random variables with geometric distributions.

\subsubsection{Case $p_1 < q_1$ and $p_2\geq q_2$}
Here, both $c_2(n)$ and $\phi_2(n,r_2)$ do not depend on $n$, whereas
$c_1(1)=\rho_1c_1(0)$. For $n$ even, the limit \eqref{eq:horribleexpressionforthelimiteven} becomes
\begin{equation}
\label{eq:case<geq}
\lim_{n\to\infty} \PP(N_n^1=r_1, N_n^2 =r_2 \mid  \tau>n) = \left(\frac{\phi_1(0,r_1)+\rho_1\phi_1(1,r_1)}{1+\rho_1}\right)  \phi_2(0,r_2).
\end{equation}
In this case, the limiting distribution is still a pair of independent random variables.

\subsubsection{Case $p_1<q_1$ and $p_2<q_2$}
In this case, the sole simplification arises from the relation $c_i(1)=\rho_i c_i(0)$. When $n$ is even, the limit \eqref{eq:horribleexpressionforthelimiteven} becomes
\begin{equation}
\label{eq:case<<}
\lim_{n\to\infty}  \PP(N_n^1=r_1, N_n^2 =r_2 \mid  \tau>n) =  \frac{\phi_1(0,r_1)\phi_2(0,r_2)+\rho_1\rho_2\phi_1(1,r_1)\phi_2(1,r_2)}{1+\rho_1\rho_2}.
\end{equation}
The limit law is that of two non-independent random variables, revealing a substantial change in behavior compared with the other drift regimes.
\end{proof}

\subsection{Proof of Theorem~\ref{thm:2dexcursioncase}}

\begin{proof}
The usual decomposition argument of Section~\ref{sec:pathdecompandindependence} gives
\begin{align*}
\PP(&N_n^1=r_1, N_n^2= r_2, \tau> n, S_n=0 ) \\
 & = \sum_{k=0}^n\PP(\widetilde{N}^1_k=r_1, \widetilde{\tau}_1>k, \tS^1_k=0, \widetilde{N}^2_{n-k}=r_2, \widetilde{\tau}_2>n-k, \tS^2_{n-k}=0, H_n=k)\\
   & = \sum_{k=0}^n\PP(\widetilde{N}^1_k=r_1, \widetilde{\tau}_1>k, \tS^1_k=0)\PP( \widetilde{N}^2_{n-k}=r_2, \widetilde{\tau}_2>n-k, \tS^2_{n-k}=0)\binom{n}{k}h_1^k h_2^{n-k}.
\end{align*}
Fix $\alpha$ and $\beta$ such that $0<\alpha<h_1<\beta<1$. Then we know from Lemma~\ref{lem:binomial_tails} that 
$$\sum_{k<\alpha n \mbox{\small\  or } k>\beta n}\binom{n}{k} h_1^k h_2^{n-k}=O(\nu^n),$$
for some $\nu\in (0,1)$.
Therefore
\begin{equation}
\label{eq:remainderconroltheorem4}
\PP(N_n^1=r_1, N_n^2=r_2, \tau>n, S_n=0 )  = F(n,r_1, r_2) + O(\nu^n),
\end{equation}
where
\begin{multline*}
F(n,r_1, r_2)=\\
\sum_{\alpha n\leq k\leq \beta n}\PP(\widetilde{N}^1_k=r_1, \widetilde{\tau}_1>k, \tS^1_k=0)\PP( \widetilde{N}^2_{n-k}=r_2, \widetilde{\tau}_2>n-k, \tS^2_{n-k}=0)\binom{n}{k}h_1^k h_2^{n-k}.
\end{multline*}
According to Theorem~\ref{thm:excursionconvergencecase}, as $n\to\infty$ through $2\NN$, we have
$$\PP(\widetilde{N}^i_n=r_i, \widetilde{\tau}_i>n, \tS^i_n=0)\sim g(r_i) \PP(\widetilde{\tau}_i>n, \tS^i_n=0),$$
where $g(r)=r/2^{r+1}$. Hence, as $n\to\infty$ through $2\NN$, 
\begin{align*}
& F(n,r_1, r_2)\\
& \sim g(r_1)g(r_2)\sum_{\substack{\alpha n\leq k\leq \beta n\\ k\in 2\NN_0}}\PP(\widetilde{\tau}_1>k, \tS^1_k=0)\PP(\widetilde{\tau}_2>n-k, \tS^2_{n-k}=0)\binom{n}{k}h_1^k h_2^{n-k}\\
    &\sim g(r_1)g(r_2)\left(\sum_{\substack{0\leq k\leq n\\ k\in 2\NN_0}}\PP(\widetilde{\tau}_1>k, \tS^1_k=0)\PP(\widetilde{\tau}_2>n-k, \tS^2_{n-k}=0)\binom{n}{k}h_1^k h_2^{n-k}+O(\nu^n)\right)\\
     &\sim g(r_1)g(r_2)\PP(\tau>n, S_n=0)+O(\nu^n).
\end{align*}
Now \emph{assume $p_1=q_1$ and $p_2=q_2$}. The estimate \eqref{eq:1dzeroattimenandtausupn} in Lemma~\ref{lem:onedzerodriftbasicestimates} leads to the simple lower bound 
\begin{align*}
\PP(\tau>n, S_n=0) &=  \sum_{k=0}^n\PP(\widetilde{\tau}_1>k, \tS^1_k=0)\PP(\widetilde{\tau}_2>n-k, \tS^2_{n-k}=0)\binom{n}{k}h_1^k h_2^{n-k}\\
& \geq \frac{C}{n^3}\sum_{k=0}^n \binom{n}{k}h_1^k h_2^{n-k} = \frac{C}{n^3}.
\end{align*}
This implies that $\nu^n=o\left(\PP(\tau>n, S_n=0)\right)$, and therefore
$$ F(n,r_1, r_2) \sim g(r_1)g(r_2)\PP(\tau>n, S_n=0).$$
In view of \eqref{eq:remainderconroltheorem4}, it follows that
$$\PP(N_n^1=r_1, N_n^2=r_2 \mid  \tau>n, S_n=0) \sim g(r_1)g(r_2).$$
This proves the theorem in the case $p_1=q_1$ and $p_2=q_2$.
But this convergence also holds in the general case since the probability under consideration does not depend on the parameters $p_1, p_2$. (See the end of the proof of Theorem~\ref{thm:2dbridgecase}.)
\end{proof}

\appendix
\section{Uniform asymptotic equivalence of sums}
\label{sec:asymptoticequiofsums}

We first gather four simple results on asymptotic equivalence, for which we could not find any reference.

\begin{lem}
\label{lem:produitdeconvolution}
Let $\left(a_n\right)_{n\geq 0}$ and $\left(b_n\right)_{n\geq 0}$ be sequences of positive numbers such that $a_n\sim an^{-\alpha}$
and $b_n \sim b n^{-\alpha}$ for some $\alpha>1$ and $a,b>0$. Then
$$\sum_{k=0}^n a_kb_{n-k}\sim \frac{Ab+aB}{n^{\alpha}},$$
where $A=\sum_{n=0}^\infty a_n$ and $B=\sum_{n=0}^\infty b_n$.
\end{lem}
\begin{proof}
The general case  follows immediately from the case $a=b=1$, which we now prove.
Let $0<t<s<1$ be fixed and  $C_n=\{k\in \NN_0 \mid  k\leq tn\}$. For all $k\in C_n$, we have $n-k\geq (1-t)n$, so that
$a_kb_{n-k}\sim \frac{a_k}{(n-k)^{\alpha}}$ as $n\to\infty$ and $k\in C_n$. Therefore, as $n\to\infty$,
$$\sum_{k\in C_n} a_k b_{n-k}\sim \sum_{k\in C_n} \frac{a_k}{(n-k)^{\alpha}}= \frac{1}{n^{\alpha}} \sum_{k=0}^{\lfloor tn\rfloor}\frac{a_k}{(1-\frac{k}{n})^{\alpha}}.$$
The last sum converges to $A$ as $n\to\infty$ by the dominated convergence theorem. Therefore,
$$\sum_{k=0}^{\lfloor tn\rfloor} a_k b_{n-k}\sim \frac{A}{n^{\alpha}}.$$
The same argument shows that
$$\sum_{k=\lfloor sn\rfloor }^{n} a_k b_{n-k}\sim \frac{B}{n^{\alpha}}.$$
Now, for the last part of the sum, we have
\begin{equation*}
\sum_{tn<k<sn} a_k b_{n-k} \sim \sum_{tn<k<sn} \frac{1}{k^{\alpha}(n-k)^{\alpha}}
 \sim \frac{1}{n^{2\alpha-1}}\int_t^s \frac{dx}{x^{\alpha}(1-x)^{\alpha}},
\end{equation*}
where we have used Lemma~\ref{lem:uniformriemannsum}, proved below. Since $\frac{1}{n^{2\alpha-1}}=o(\frac{1}{n^{\alpha}})$ for $\alpha>1$, the lemma follows.
\end{proof}

The next two lemmas deal with uniform equivalence for the remainders of sums or integrals.

\begin{lem}
\label{uniform_equivalence_of_remainders}
Let $f_r, g_r:\NN\to\RR^+$ be non-negative functions, where $r$ belongs to some set $I$ of indices. Let also $(I_n)$ be a non-decreasing sequence of subsets of $I$.
Assume that the following conditions hold:
\begin{enumerate}
\item $f_r(n)\sim g_r(n)$ as $n\to\infty$, uniformly for $r\in I_n$;
\item $\sum_{k>n} f_r(k)$ is finite for all $r\in I_n$.
\end{enumerate}
Then $\sum_{k>n} g_r(k)$ is finite for all $r\in I_n$, and
$$\sum_{k>n} f_r(k)\sim \sum_{k>n} g_r(k)$$ as $n\to\infty$, uniformly for $r\in I_n$.
\end{lem}
\begin{proof}
Follows from the definitions.
\end{proof}

\begin{lem}
\label{uniform_equivalence_of_sums_and_integral}
Let $f_r:\RR\to\RR^+$ be a non-negative measurable function, where $r$ belongs to some set $I$ of indices. Let also $(I_n)$ be a non-decreasing sequence of subsets of $I$.
Assume that the following conditions hold:
\begin{itemize}
\item $\max_{t\in [n-1,n]}f_r(t)\sim \min_{t\in [n-1,n]}f_r(t)$ as $n\to\infty$, uniformly for $r\in I_n$;
\item $\sum_{k>n} f_r(k)$ is finite for all $r\in I_n$.
\end{itemize}
Then $\int_n^{\infty} f_r(t)\,dt$ is finite for all $r\in I_n$, and
$$\sum_{k>n} f_r(k)\sim \int_n^{\infty} f_r(t)\,dt$$ as $n\to\infty$, uniformly for $r\in I_n$.
\end{lem}
\begin{proof}
Set $M_r(n)=\max_{t\in [n-1,n]}f_r(t)$ and $m_r(n)=\min_{t\in [n-1,n]}f_r(t)$.
Since $$m_r(n)\leq f_r(n) \leq M_r(n)$$
and
$$m_r(n)\leq \int_{n-1}^{n} f_r(t)\,dt \leq M_r(n),$$
the first condition implies that 
$f_r(n)\sim \int_{n-1}^{n} f_r(t) \,dt$ as $n\to\infty$, uniformly for $r\in I_n$.
The result follows directly from Lemma~\ref{uniform_equivalence_of_remainders}.
\end{proof}

The first assumption in Lemma~\ref{uniform_equivalence_of_sums_and_integral} is well suited to our setting, since the asymptotic quantities under consideration exhibit polynomial decay. The last lemma deals with uniform Riemann sums.

\begin{lem}
\label{lem:uniformriemannsum}
Let $f:[0,T]\to\RR$ be a continuous and differentiable function with a bounded derivative. For any sequence $(\ell_n)$ of positive real numbers such that $\ell_n\to\infty$, 
$$\lim_{n\to\infty} \sum_{k=1}^{\theta \ell_n}\frac{1}{\ell_n}f\left(\frac{k}{\ell_n}\right)=\int_0^{\theta}f(t)\,dt,$$
where the convergence holds uniformly for $\theta\leq T$.
\end{lem}
\begin{proof}
Let $B$ be an upper bound for $\vert f'\vert$. Then $\vert f(y)-f(x)\vert\leq B \vert x-y\vert$ for all $x,y\in\RR$. Therefore 
$$\left \vert \frac{1}{\ell_n} f\left(\frac{k}{\ell_n}\right)-\int_{\frac{k-1}{\ell_n}}^{\frac{k}{\ell_n}} f(t)\,dt \right\vert\leq 
\int_{\frac{k-1}{\ell_n}}^{\frac{k}{\ell_n}} \left\vert f\left(\frac{k}{\ell_n}\right)-f(t)\right\vert\,dt \leq \frac{B}{\ell^2_n}.$$
Summing over $k\in\{1, \ldots ,\lfloor \theta \ell_n\rfloor\}$, we obtain
$$\left\vert \sum_{k=1}^{\theta \ell_n}\frac{1}{\ell_n}f\left(\frac{k}{\ell_n}\right)-\int_0^{\frac{\lfloor \theta\ell_n\rfloor}{\ell_n}}f(t)\,dt\right\vert \leq \frac{\theta B}{\ell_n}.$$
On the other hand, we have
$$\left\vert\int_{\frac{[\theta\ell_n]}{\ell_n}}^{\theta} f(t) \,dt\right\vert \leq \frac{1}{\ell_n}\sup_{[0,T]}\vert f(t)\vert,$$
hence 
$$\left\vert \sum_{k=1}^{\theta \ell_n}\frac{1}{\ell_n}f\left(\frac{k}{\ell_n}\right)-\int_0^{\theta}f(t)\,dt\right\vert\leq \frac{T B +\sup_{[0,T]}\vert f(t)\vert}{\ell_n},$$
and the result is proved.
\end{proof}

We conclude the appendix with the following lemma on binomial tails, which plays a role in the proofs of Theorems~\ref{thm:2dunconditioned}, \ref{thm:2dbridgecase}, \ref{thm:2dmeandercase}, and \ref{thm:2dexcursioncase}.
Although it follows readily from the classical Chernoff bounds, we include a short elementary proof for the reader's convenience.
\begin{lem}
\label{lem:binomial_tails}
Let $X$ be a binomial random variable with parameters $n\geq 1$ and $p\in (0,1)$. For any $\alpha <p<\beta$, there exists $\nu\in (0,1)$ such that
$$\PP(X\notin [n\alpha, n\beta])=O(\nu^n).$$
\end{lem}
\begin{proof}
For every $t>0$, Markov's inequality gives
$$\PP( X \geq \beta n)
 \leq e^{-t\beta n}\EE(e^{tX})
 =\left(e^{-t\beta}(1-p+pe^t)\right)^n.$$
Let $g(t)=e^{-t\beta}(1-p+pe^t)$. Since $g(0)=1$ and $g'(0)=p-\beta <0$, the function $g$ reaches on $(0,\infty)$ a value strictly smaller than $1$. Hence there exists
$\nu_1<1$ such that
$$ \PP( X \geq \beta n)\leq \nu_1^n.$$
Similarly, since $\alpha<p$, we have
$\PP( X \leq \alpha n)\leq \nu_2^n$
for some $\nu_2<1$. Therefore, denoting by $\nu=\max\{\nu_1,\nu_2\}<1$,
\begin{equation*} \PP(X\notin [n\alpha, n\beta]) \leq 2 \nu^n.\qedhere\end{equation*}
\end{proof}

\section*{Acknowledgments}
The authors would like to thank Nicholas Beaton, Denis Denisov, Nikita Elizarov and Michael Wallner for interesting discussions at an early stage of the project. They also extend their sincere thanks to the two referees for their thorough and constructive reports.

\bibliographystyle{abbrv}  % or another style like alpha, abbrv, etc.
\bibliography{Returns.bib}  % without the .bib extension

@article{BaFl-02,
 author = {Banderier, Cyril and Flajolet, Philippe},
 title = {Basic analytic combinatorics of directed lattice paths},
 fjournal = {Theoretical Computer Science},
 journal = {Theor. Comput. Sci.},
 issn = {0304-3975},
 volume = {281},
 number = {1-2},
 pages = {37--80},
 year = {2002},
 language = {English},
 doi = {10.1016/S0304-3975(02)00007-5},
 zbMATH = {1767887},
 Zbl = {0996.68126}
}

@InProceedings{BaKuWaWa-24,
  author =	{Banderier, Cyril and Kuba, Markus and Wagner, Stephan and Wallner, Michael},
  title =	{{Composition Schemes: q-Enumerations and Phase Transitions in Gibbs Models}},
  booktitle =	{35th International Conference on Probabilistic, Combinatorial and Asymptotic Methods for the Analysis of Algorithms (AofA 2024)},
  pages =	{7:1--7:18},
  series =	{Leibniz International Proceedings in Informatics (LIPIcs)},
  ISBN =	{978-3-95977-329-4},
  ISSN =	{1868-8969},
  year =	{2024},
  volume =	{302},
  editor =	{Mailler, C\'{e}cile and Wild, Sebastian},
  publisher =	{Schloss Dagstuhl -- Leibniz-Zentrum f{\"u}r Informatik},
  address =	{Dagstuhl, Germany},
  URL =		{https://drops.dagstuhl.de/entities/document/10.4230/LIPIcs.AofA.2024.7},
  URN =		{urn:nbn:de:0030-drops-204427},
  doi =		{10.4230/LIPIcs.AofA.2024.7}
}

@article{BaKuWa-24,
 author = {Banderier, Cyril and Kuba, Markus and Wallner, Michael},
 title = {Phase transitions of composition schemes: {Mittag}-{Leffler} and mixed {Poisson} distributions},
 fjournal = {The Annals of Applied Probability},
 journal = {Ann. Appl. Probab.},
 issn = {1050-5164},
 volume = {34},
 number = {5},
 pages = {4635--4693},
 year = {2024},
 language = {English},
 doi = {10.1214/24-AAP2076},
 zbMATH = {7927498},
 Zbl = {1559.60195}
}

@incollection{BaWa-14,
 author = {Banderier, Cyril and Wallner, Michael},
 title = {Some reflections on directed lattice paths},
 booktitle = {Proceeding of the 25th international conference on probabilistic, combinatorial and asymptotic methods in the analysis of algorithms, AofA'14, UPMC-Jussieu, Paris, France, June 16--20, 2014},
 pages = {25--36},
 year = {2014},
 publisher = {Nancy: The Association. Discrete Mathematics \& Theoretical Computer Science (DMTCS)},
 language = {English},
 zbMATH = {6547927},
 Zbl = {1332.68162}
}

@article{BeOwRe-19,
 author = {Beaton, N. R. and Owczarek, A. L. and Rechnitzer, A.},
 title = {Exact solution of some quarter plane walks with interacting boundaries},
 fjournal = {The Electronic Journal of Combinatorics},
 journal = {Electron. J. Comb.},
 issn = {1077-8926},
 volume = {26},
 number = {3},
 pages = {research paper p3.53, 39},
 year = {2019},
 language = {English},
 url = {www.combinatorics.org/ojs/index.php/eljc/article/view/v26i3p53},
 zbMATH = {7110970},
 Zbl = {1420.05010}
}

@article{DeWa15,
 author = {Denisov, Denis and Wachtel, Vitali},
 title = {Random walks in cones},
 fjournal = {The Annals of Probability},
 journal = {Ann. Probab.},
 issn = {0091-1798},
 volume = {43},
 number = {3},
 pages = {992--1044},
 year = {2015},
 language = {English},
 doi = {10.1214/13-AOP867},
 zbMATH = {6455727},
 Zbl = {1332.60066}
}

@article{Du-14,
 author = {Duraj, Jetlir},
 title = {Random walks in cones: the case of nonzero drift},
 fjournal = {Stochastic Processes and their Applications},
 journal = {Stochastic Processes Appl.},
 issn = {0304-4149},
 volume = {124},
 number = {4},
 pages = {1503--1518},
 year = {2014},
 language = {English},
 doi = {10.1016/j.spa.2013.12.003},
 zbMATH = {6261889},
 Zbl = {1319.60088}
}

@misc{Fel68,
 author = {Feller, W.},
 title = {An introduction to probability theory and its applications. {I}},
 year = {1968},
 language = {English},
 howpublished = {New {York}-{London}-{Sydney}: {John} {Wiley} and {Sons}, {Inc}. {XVIII}, 509 p. (1968).},
 zbMATH = {3249395},
 Zbl = {0155.23101}
}

@book{Fel71,
 author = {Feller, William},
 title = {An introduction to probability theory and its applications. {Vol}. {II}. 2nd ed.},
 fseries = {Wiley Series in Probability and Mathematical Statistics},
 series = {Wiley Ser. Probab. Math. Stat.},
 year = {1971},
 publisher = {John Wiley \& Sons, Hoboken, NJ},
 language = {English},
 zbMATH = {3349081},
 Zbl = {0219.60003}
}

@book{FlaSed09,
 author = {Flajolet, Philippe and Sedgewick, Robert},
 title = {Analytic combinatorics},
 isbn = {978-0-521-89806-5},
 year = {2009},
 publisher = {Cambridge: Cambridge University Press},
 language = {English},
 zbMATH = {5485323},
 Zbl = {1165.05001}
}

@book{Spi-76,
 author = {Spitzer, Frank},
 title = {Principles of random walk. 2nd ed},
 fseries = {Graduate Texts in Mathematics},
 series = {Grad. Texts Math.},
 issn = {0072-5285},
 volume = {34},
 year = {1976},
 publisher = {Springer, Cham},
 language = {English},
 zbMATH = {3560403},
 Zbl = {0359.60003}
}

@article{TaOwRe-16,
 author = {Tabbara, R. and Owczarek, A. L. and Rechnitzer, A.},
 title = {An exact solution of three interacting friendly walks in the bulk},
 fjournal = {Journal of Physics A: Mathematical and Theoretical},
 journal = {J. Phys. A, Math. Theor.},
 issn = {1751-8113},
 volume = {49},
 number = {15},
 pages = {27},
 note = {Id/No 154004},
 year = {2016},
 language = {English},
 doi = {10.1088/1751-8113/49/15/154004},
 zbMATH = {6610543},
 Zbl = {1342.82072}
}

@incollection{Wa-16,
 author = {Wallner, Michael},
 title = {A half-normal distribution scheme for generating functions and the unexpected behavior of {Motzkin} paths},
 booktitle = {Proceedings of the 27th international conference on probabilistic, combinatorial and asymptotic methods for the analysis of algorithms -- AofA'16, Krak\'ow, Poland, July 4--8, 2016},
 pages = {12},
 year = {2016},
 publisher = {Krak{\'o}w: Jagiellonian University, Department of Theoretical Computer Science},
 language = {English},
 zbMATH = {7048745},
 Zbl = {1409.05020}
}

@article{Wa-20,
 author = {Wallner, Michael},
 title = {A half-normal distribution scheme for generating functions},
 fjournal = {European Journal of Combinatorics},
 journal = {Eur. J. Comb.},
 issn = {0195-6698},
 volume = {87},
 pages = {20},
 note = {Id/No 103138},
 year = {2020},
 language = {English},
 doi = {10.1016/j.ejc.2020.103138},
 zbMATH = {7198674},
 Zbl = {1439.05018}
}

@article{BeDo-16,
 author = {Beauregard, Raymond A. and Dobrushkin, Vladimir A.},
 title = {Multisection of series},
 fjournal = {The Mathematical Gazette},
 journal = {Math. Gaz.},
 issn = {0025-5572},
 volume = {100},
 number = {549},
 pages = {460--470},
 year = {2016},
 language = {English},
 doi = {10.1017/mag.2016.111},
 zbMATH = {6790103},
 Zbl = {1384.05012}
}

@misc{Cra-38,
 author = {Cram{\'e}r, H.},
 title = {Sur un nouveau th{\'e}or{\`e}me-limite de la th{\'e}orie des probabilit{\'e}s.},
 year = {1938},
 language = {French},
 howpublished = {Actual. sci. industr. 736, 5-23. ({Conf{\'e}r}. internat. {Sci}. math. {Univ}. {Gen{\`e}ve}. {Th{\'e}orie} des probabilit{\'e}s. {III}: {Les} sommes et les fonctions de variables al{\'e}atoires.) (1938).},
 zbMATH = {2516150},
 JFM = {64.0529.01}
}

\end{document}